\documentclass[11pt]{amsart}%
\usepackage[dvips]{graphicx}
\usepackage{caption}
\usepackage{amsmath}
\usepackage{color}
\usepackage{amsfonts}
\usepackage{amssymb}%
\usepackage[all]{xy}
\usepackage[left=1in, right=1in, top=1in, bottom=1in]{geometry}
\newlength{\cellsize}
\newcommand\tableau[1]{
\vcenter{
\let\\=\cr
\baselineskip=-16000pt
\lineskiplimit=16000pt
\lineskip=0pt
\halign{&\tableaucell{##}\cr#1\crcr}}}

\newcommand{\tableaucell}[1]{{%
\def \arg{#1}\def \void{}%
\ifx \void \arg
\vbox to \cellsize{\vfil \hrule width \cellsize height 0pt}%
\else
\unitlength=\cellsize
\begin{picture}(1,1)
\put(0,0){\makebox(1,1){$#1$}}
\put(0,0){\line(1,0){1}}
\put(0,1){\line(1,0){1}}
\put(0,0){\line(0,1){1}}
\put(1,0){\line(0,1){1}}
\end{picture}%
\fi}}

\newtheorem{theorem}{Theorem}[section]

\newtheorem{corollary}[theorem]{Corollary}

\newtheorem{definition}[theorem]{Definition}
\newtheorem{example}[theorem]{Example}

\newtheorem{lemma}[theorem]{Lemma}

\newtheorem{proposition}[theorem]{Proposition}
\newtheorem{remark}[theorem]{Remark}
\newtheorem{remarks}[theorem]{Remarks}

\begin{document}
\title{Atomic decomposition of characters and crystals}

\author[C.~Lecouvey]{C\'edric Lecouvey}
\address[C\'edric Lecouvey]{Institut Denis Poisson, 
Facult\'e des Sciences et Techniques, Universit\'e de Tours, Parc de Grandmont, 37200 Tours, France}
\email{cedric.lecouvey@lmpt.univ-tours.fr}
\urladdr{https://www.idpoisson.fr/lecouvey/}

\author[C.~Lenart]{Cristian Lenart}
\address[Cristian Lenart]{Department of Mathematics and Statistics, State University of New York at Albany, 
1400 Washington Avenue, Albany, NY 12222, U.S.A.}
\email{clenart@albany.edu}
\urladdr{http://www.albany.edu/\~{}lenart/}


\begin{abstract}
Lascoux stated that the type $A$ Kostka-Foulkes polynomials $K_{\lambda,\mu}(t)$ expand positively in terms of so-called atomic polynomials. For any semisimple Lie algebra, the former polynomial is a $t$-analogue of the multiplicity of the dominant weight $\mu$ in the irreducible representation of highest weight $\lambda$. We formulate the atomic decomposition in arbitrary type, and view it  as a strengthening of the monotonicity of $K_{\lambda,\mu}(t)$. We also define a combinatorial version of the atomic decomposition, as a decomposition of a modified crystal graph.  We prove that this stronger version holds in type $A$ (which provides a new, conceptual approach to Lascoux's statement), in types $B$, $C$, and $D$ in a stable range for $t=1$, as well as in some other cases, while we conjecture that it holds more generally. Another conjecture stemming from our work leads to an efficient computation of $K_{\lambda,\mu}(t)$. We also give a geometric interpretation.   
\end{abstract}

\maketitle

\section{Introduction}\label{Intro}

The starting point of this paper is a result of Lascoux on the (type $A$)
Kostka-Foulkes polynomials $K_{\lambda,\mu}(t)$, which are well-known
$t$-analogues of the Kostka numbers $K_{\lambda,\mu}$, i.e., the number of
semistandard Young tableaux of shape $\lambda$ and content $\mu$. {Lascoux
\cite{lascpw} stated the decomposition of the Kostka-Foulkes
polynomials into so-called \emph{atomic polynomials}. Some arguments of the
proof in \cite{lascpw} remained elusive}, and it was not until the work of
Shimozono \cite{shimam} that the type $A$ atomic decomposition was
{completely} accepted, this time in larger generality (for the so-called
\emph{generalized Kostka-Foulkes polynomials}). However, the latter proof
involves several intricate combinatorial arguments and related concepts, such
as \emph{plactic monoid}, \emph{cyclage}, and \emph{catabolism}. The
main goal of this paper is to provide a simpler, more conceptual approach, which
has the additional advantage of extending beyond type $A$.

Lusztig defined a remarkable $t$-analogue $K_{\lambda,\mu}(t)$ of the
multiplicity of a weight $\mu$ in the irreducible representation with highest
weight $\lambda$ of a semisimple Lie algebra \cite{Lut}. For dominant weights
$\mu$, these polynomials generalize the type $A$ ones mentioned above, and are
therefore also called \emph{Kostka-Foulkes polynomials}. They have remarkable
properties: 
\begin{itemize}
\item they are essentially \emph{affine Kazhdan-Lusztig polynomials}
\cite{Kat};
\item they record the \emph{Brylinski filtration} of weight spaces \cite{bry};
\item they are the coefficients in the expansion of an irreducible character
in terms of \emph{Hall-Littlewood polynomials} \cite{macsfg};
\item they are related to the \emph{energy function} coming from solvable
lattice models \cite{LecShi,NY}.
\end{itemize}
In classical Lie types, when the rank increases, these polynomials exhibit a
stabilization property, so they have \emph{stable versions} \cite{Lec4}. We
refer to \cite{NR,Ste} for more information on Kostka-Foulkes polynomials. 

There are two combinatorial formulas for type $A$ Kostka-Foulkes polynomials:
one due to Lascoux-Sch\"{u}tzenberger, in terms of the \emph{charge statistic}
on semistandard Young tableaux \cite{LSc1}, and one in terms of the
corresponding Kashiwara \emph{crystal graphs} \cite{HK,kascbm,kash}, due to
Lascoux-Leclerc-Thibon \cite{LLT}. Similar partial combinatorial descriptions
in types $B-D$, in terms of the corresponding \emph{Kashiwara-Nakashima
tableaux} \cite{HK}, and a conjectured charge statistic in type $C$ were constructed by
the first author in \cite{lec,Lec2}. In our previous paper \cite{LL}, we prove
the first general formula beyond type $A$, namely a formula for $K_{\lambda
,0}(t)$ of type $C$. This extends the Lascoux-Leclerc-Thibon formula by using
a simpler approach, and can be expressed in terms of \emph{King tableaux}. 

In Section~\ref{sec:atdec}, following Lascoux \cite{lascpw}, we formulate the
\emph{$t$-atomic decomposition} property in arbitrary Lie type as a
nonnegative expansion for both a Kostka-Foulkes polynomial, and a $t$-analogue
$\chi_{\lambda}^{+}(t)$ of the dominant part of an irreducible character
(defined in terms of Kostka-Foulkes polynomials). Here a character expansion
is in terms of so-called \emph{layer sum polynomials}, which record all
weights of some irreducible representation with multiplicity $1$. The $t$-atomic
decomposition property is a strengthening of the monotonicity of
Kostka-Foulkes polynomials, which holds in the full generality of affine
Kazhdan-Lusztig polynomials \cite[Corollary~3.7]{BrM}. 

For $t=1$, the atomic decomposition was also considered from a purely
algebraic point of view in \cite{khare, ras, Schu}. Some partial results were
given, for instance for the inverse expansion (of layer sum polynomials in
terms of characters). However, the atomic decomposition itself is less
understood, even for $t=1$; for instance, it does not always exist, in the
sense that the mentioned character expansion sometimes fails to be
nonnegative, contrary to what was claimed in Theorem~2.2 of \cite{Schu} (see
Example \ref{counterex}). Nevertheless, these failures seem to be mild. 

As opposed to the algebraic approach mentioned above, in
Section~\ref{sec:tatdec} we define a stronger $t$-atomic decomposition
property, at the combinatorial level of the highest weight crystal
$B(\lambda)$. This property involves a partition of the dominant part
$B(\lambda)^{+}$ of $B(\lambda)$, and a statistic $\mathrm{c}(\cdot)$ on
$B(\lambda)^{+}$. The combinatorial $t$-atomic decomposition leads to combinatorial formulas
for both the Kostka-Foulkes polynomials and the atomic polynomials (into which
the former decompose). 

In Section \ref{sec:inf}, we consider the case when $\lambda$ goes to
infinity, in types $A_{n-1}$, $B_{n}$, $C_{n}$, $D_{n}$, and $G_{2}$. We
introduce the notion of a $t$-atomic decomposition of the crystal $B(\infty)$,
and discuss how it can be realized. 

Next we introduce the two main ingredients for constructing a $t$-atomic
decomposition of a finite highest weight crystal: the
\emph{partial order on dominant weights}, and a \emph{modified crystal graph}
structure. The natural poset structure on dominant weights is discussed in
Section~\ref{sec:pod}, by recalling some important properties which hold in
arbitrary Lie type \cite{stepod}. We need extra information about this poset,
namely the structure of its small intervals, which was only known in type $A$
\cite{brylip}. The second ingredient, namely a {modified crystal graph}
structure on $B(\lambda)^{+}$, is discussed in Section~\ref{sec:modcr}; this
structure can also be viewed as a poset, with covers corresponding to the
modified crystal edges. The associated modified crystal operators are obtained
mainly by conjugating the ordinary Kashiwara operator $\tilde{f}_{1}$ under the
crystal action of the Weyl group. There are two main differences between the
original and the modified crystal operators: (1)~the latter are indexed by
arbitrary positive roots; (2)~the original $B(\lambda)$ is a connected
graph/poset, whereas the modified crystal graph on $B(\lambda)^{+}$ is disconnected in general.
An interesting question is whether the modified crystal operators for $\lambda$ become those for $B(\infty)$ in Section~\ref{sec:inf} when $\lambda$ goes to infinity.

In Section~\ref{sec:proofatdec}, we start by showing that type $A$ crystals
admit a $t$-atomic decomposition, thus realizing combinatorially the classical
result, while also providing a relatively simple and conceptual proof of it.
Here the desired partition of the modified crystal poset $B(\lambda)^{+}$ is given by its components, which are shown to admit unique minimal and maximal vertices. 
The relevant statistic $\mathrm{c}(\cdot)$ is the
Lascoux-Sch\"{u}tzenberger charge \cite{LSc1}. Similarly, given a fixed
dominant weight $\lambda$ of type $B$, $C$, or $D$, and assuming a large enough
rank (depending on $\lambda$), we show that the corresponding modified crystal graph 
$B(\lambda)^{+}$ gives an atomic decomposition (i.e., in the case $t=1$). Note that, in type $B$, we need to add to the modified crystal graph the operators obtained by conjugating the Kashiwara operator~$\tilde{f}_n$. 

In Section \ref{sec:conj} we conjecture that the mentioned decompositions in types $B$, $C$, and $D$ are, in
fact, $t$-atomic decompositions, for appropriate choices of the statistic
$\mathrm{c}(\cdot)$, which are related to the results in \cite{lec,Lec2} and
\cite{LL}. We also discuss the atomic decomposition of the stable one-dimensional sums in classical types, and explain how to obtain a $t$-atomic decomposition of the
crystal $B(\widetilde{\alpha})^{+}$ of the adjoint representation of
$\mathfrak{g}$ in any type. The latter fact is useful because it clarifies the type and rank
restrictions we considered in Section~\ref{sec:proofatdec}: for instance, in types $C$ and $D$, they are needed to ensure
that the covering relations in the poset $B(\lambda)^{+}$ make only
appear roots in the orbit of the simple root $\alpha_{1}$. Under such 
assumptions, we are indeed able to establish fine commutation relations
satisfied by the modified crystal operators, which are required to derive the
atomic decomposition of crystals (notably, the existence of unique maximal and minimal vertices). For the simply laced types $D$ (with no
large rank assumption) and $E$, the atomic decomposition of the characters can
only hold with some restrictions on the dominant weights considered, but we
expect that it can be derived similarly from the same modified crystal operators
and a detailed analysis of the corresponding dominant weight poset. In the
non-simply laced case of general rank, the situation becomes more complicated,
but it is certainly possible to again derive relevant atomic decompositions of
crystals by using modified crystal operators; this time, they would be defined using the Weyl group
conjugation of two ordinary crystal operators, associated to one simple root of
each length. 

A geometric interpretation of Lascoux's atomic decomposition of the type $A$
Kostka-Foulkes polynomial was given in \cite{bahnov} in terms of
\emph{nilpotent orbit varieties}. In Section~\ref{sec:geom}, we provide a
different, type-independent interpretation of an atomic decomposition of
$\chi_{\lambda}^{+}$, in terms of the \emph{geometric Satake correspondence}.

\noindent\textbf{Acknowledgments:} The second author gratefully acknowledges the partial support from the NSF grant DMS--1362627 and the Simons grant \#584738. He thanks Institut des Hautes \'Etudes Scientifiques (IH\'ES) for its hospitality during July-August 2018, when this work was completed. Both authors are grateful to Arthur Lubovsky and Adam Schultze for the computer tests (based on the {\tt Sage} \cite{combinat} system) related to this work; they also received support from the NSF grant mentioned above. 

\section{The atomic decomposition: background, definitions, and basic facts}

\subsection{Characters and $t$-deformations}\label{sec:backgr}

Let $\mathfrak{g}$ be a simple Lie algebra over $\mathbb{C}$ of rank $r$ with
triangular decomposition
\[
\mathfrak{g=}\bigoplus\limits_{\alpha\in R^{+}}\mathfrak{g}_{\alpha}%
\oplus\mathfrak{h}\oplus\bigoplus\limits_{\alpha\in R^{+}}\mathfrak{g}%
_{-\alpha}\,,
\]
so that $\mathfrak{h}$ is the Cartan subalgebra of $\mathfrak{g}$ and $R^{+}$
its set of positive roots. The root system $R=R^{+}\sqcup(-R^{+})$ of
$\mathfrak{g}$ is realized in a real Euclidean space $E$ of dimension $r$ with
inner product $\langle\,\cdot\,,\,\cdot\,\rangle$. For any $\alpha\in R,$ we write $\alpha^{\vee
}=\frac{2\alpha}{\langle\alpha,\alpha\rangle}$ for its coroot. Let $S\subset R^{+}$ be the
subset of simple roots and $Q^{+}$ the {$\mathbb{Z}_+$}-cone generated by $S$. The set $P$ of
integral weights for $\mathfrak{g}$ consists of elements $ \lambda$ satisfying $\langle \lambda,\alpha^{\vee}%
\rangle\in\mathbb{Z}$ for any $\alpha\in R$. We write
$P^{+}=\{\lambda\in P\mid\langle\lambda,\alpha^{\vee}\rangle\geq0$ for any $\alpha\in S\}$ for
the cone of dominant weights of $\mathfrak{g}$, and denote by $\omega
_{1},\ldots,\omega_{r}$ its fundamental weights. We consider the group algebra of the weight lattice, denoted ${\mathbb Z}[P]$, which has a ${\mathbb Z}$-basis of formal exponentials $e^\lambda$, for $\lambda\in P$, with usual multiplication $e^\lambda e^\mu=e^{\lambda+\mu}$. Let $W$ be the Weyl group of
$\mathfrak{g}$ generated by the reflections $s_{\alpha}$ with $\alpha\in S$,
and write $\ell(\cdot)$ for the corresponding length function. The {\em dominance order}
$\leq$ on $P^+$ is defined by $\lambda<\mu$ if and only if $\mu-\lambda$ decomposes as a sum of positive roots (or equivalently, simple roots) with nonnegative integer coefficients. Many interesting properties of this order were studied in \cite{stepod}, and some of them will be used in this paper; for instance, each component of this poset is a lattice.

Let $\chi_{\lambda}^{\mathfrak{g}}$ be the {\em Weyl character} associated to the
finite-dimensional irreducible representation $V(\lambda)$ of $\mathfrak{g}$
with highest weight $\lambda\in P^+$, namely%
\[
\chi_{\lambda}^{\mathfrak{g}}=\sum_{\gamma\in P}K_{\lambda,\gamma}^{\mathfrak g}\,e^{\gamma}\,,%
\]
where $K_{\lambda,\gamma}^{\mathfrak{g}}$ is the dimension of the weight space $\gamma$ in
$V(\lambda)$. For simplicity, we remove the superscript $\mathfrak{g}$ when the context makes it clear. By the {\em Weyl character formula} we have
\begin{equation}
\chi_{\lambda}=\frac{\sum_{w\in W}(-1)^{\ell(w)}e^{w(\lambda
+\rho)-\rho}}{\prod_{\alpha\in R^{+}}(1-e^{-\alpha})}\,.
\end{equation}
This formula expresses the weight multiplicities $K_{\lambda,\gamma}$ as follows:%
\[
K_{\lambda,\gamma}=\sum_{w\in W}(-1)^{\ell(w)}\emph{P}%
(w(\lambda+\rho)-(\gamma+\rho))\,,
\]
where $\rho$ is the half sum of positive roots, and $P(\cdot)$ is the {\em Kostant
partition function}, defined by
\[
\prod_{\alpha\in R^{+}}\frac{1}{1-e^{\alpha}}=\sum_{\beta\in Q^{+}}%
\emph{P}(\beta)\,e^{\beta}\,.
\]
When $\gamma=\mu$ is dominant, the multiplicity $K_{\lambda,\mu}%
$ has an interesting {\em $t$-analogue} $K_{\lambda,\mu
}(t)$, also known as a {\em Kostka-Foulkes polynomial}. This was introduced by Lusztig \cite{Lut}, who defined%
\begin{equation}
K_{\lambda,\mu}(t):=\sum_{w\in W}(-1)^{\ell(w)}\emph{P}%
_{t}(w(\lambda+\rho)-(\mu+\rho))\,;
\end{equation}
here the {\em $t$-analogue of the Kostant partition function} $P_t(\cdot)$ is given by
\[
\prod_{\alpha\in R^{+}}\frac{1}{1-te^{\alpha}}=\sum_{\beta\in Q^{+}}%
\emph{P}_{t}(\beta)\,e^{\beta}\,.
\]
We have $K_{\lambda,\mu}(1)=K_{\lambda,\mu}%
$. Moreover, $K_{\lambda,\mu}(t)$ is essentially an {\em affine
Kazhdan-Lusztig polynomial}, which implies that it has nonnegative integer coefficients. More precisely, we have
\begin{equation}\label{affinekl}
K_{\lambda,\mu}(t)=t^{\langle\lambda-\mu,\rho^\vee\rangle}P_{w_\mu,w_\lambda}(t^{-1})\,,
\end{equation}
where $w_\lambda$ denotes the longest element of $Wt_\lambda W$, and $t_\lambda$ is the translation by $\lambda$ in the extended affine Weyl group \cite{Kat} (see also \cite[Section~4]{Ste}); note that $\langle\lambda-\mu,\rho^\vee\rangle$ is the number of simple roots in the decomposition of $\lambda-\mu$, counted with multiplicity. Based on \eqref{affinekl}, we let 
\begin{equation}\label{ktilde}
\widetilde{K}_{\lambda,\mu}(t):=t^{\langle\lambda-\mu,\rho^\vee\rangle}K_{\lambda,\mu}(t^{-1})\,,\;\;\;\;\;\mbox{so $\widetilde{K}_{\lambda,\mu}(t)=P_{w_\mu,w_\lambda}(t)$}\,.
\end{equation}

\bigskip

To each irreducible representation $V(\lambda)$ is associated an abstract
{\em Kashiwara crystal} $B(\lambda)$ (see \cite{HK,kascbm,kash} for background on crystals), and
we have
\begin{equation}
\chi_{\lambda}=\sum_{b\in B(\lambda)}e^{\mathrm{wt}(b)}\,,\label{LCF}%
\end{equation}
where $\mathrm{wt}(b)$ is the weight of the vertex $b\in B(\lambda)$. The
crystal $B(\infty)$ is defined as the direct limit of the crystal $B(\lambda)$
when $\lambda$ goes to infinity in the interior of its Weyl chamber. It
corresponds to the crystal of the positive part of the quantum group
$U_{q}(\mathfrak{g})$ associated to $\mathfrak{g}$.\ One can then prove that
\begin{equation}
\mathrm{char}\,B(\infty)=\sum_{b\in B(\infty)}e^{\mathrm{wt}(b)}=\prod
_{\alpha\in R^{+}}\frac{1}{1-e^{-\alpha}}\,.
\end{equation}

\subsection{The definition of the atomic decomposition}\label{sec:atdec}

Let us denote by $P(\lambda)$ the set of weights of $V(\lambda)$, i.e., the set of $\gamma$ such that
$K_{\lambda,\gamma}>0$. Also set $P^{+}(\lambda
)=P(\lambda)\cap P^{+}$. Recall that we have $P^{+}%
(\lambda)=\{\mu\in P^{+}\mid\mu\leq\lambda\}$. Since we have $K_{\lambda
,\gamma}=K_{\lambda,w(\gamma)}$ for any $w\in
W$, the character $\chi_{\lambda}$ is completely determined by its
dominant part%
\[
\chi_{\lambda}^+:=\sum_{\mu\in P^{+}(\lambda)}\,K_{\lambda,\mu
}\,e^{\mu}\,.
\]
We also define the $t$-analogue of $\chi_{\lambda}^+$ by
\begin{equation}\label{domchar}
\chi_{\lambda}^+(t):=\sum_{\mu\in P^{+}(\lambda)}%
K_{\lambda,\mu}(t)\,e^{\mu}\,.%
\end{equation}

For any dominant weight $\mu$, define the {\em layer sum polynomials} by
\begin{equation}
w_{\mu}:=\sum_{\gamma\in P(\mu)}e^{\gamma}\;\;\;\;\text{ and }\;\;\;\;w_{\mu}^{+}%
:=\sum_{\nu\in P^{+}(\mu)}e^{\nu}=\sum_{\nu\leq\mu}e^{\nu}\,.
\end{equation}
Observe that we have
\[
w_{\mu}=\sum_{\nu\in P^{+}(\mu)}m_{\nu}\,,\;\;\;\;\text{ where }m_{\nu}=\sum
_{\gamma\in W\nu}e^{\gamma}\,.
\]
Similarly, define
\begin{equation}
w_{\mu}^{+}(t):=\sum_{\nu\in P^{+}(\mu)}t^{\langle\mu-\nu,\rho^\vee\rangle}e^{\nu}=\sum_{\nu\leq\mu}t^{\langle\mu-\nu,\rho^\vee\rangle}e^{\nu}\,.%
\end{equation}

\begin{remark}\label{bases}{\rm 
By the triangularity of the coresponding transition matrices, we have:
\begin{itemize}
\item $w_\mu$ is a ${\mathbb Z}$-basis of the $W$-invariants ${\mathbb Z}[P]^W$;
\item $w_\mu^+$ is a ${\mathbb Z}$-basis of ${\mathbb Z}[P^+]$;
\item $w_\mu^+(t)$ is a ${\mathbb Z}[t]$-basis of ${\mathbb Z}[t][P^+]$.
\end{itemize}}
\end{remark}

There exists a Weyl-type formula for the polynomials $w_{\mu}$, which follows
from Brion's formula \cite{Bri} counting the cardinality of the intersection
between a convex polygon and a lattice, see \cite[Theorem~4.3]{Pos}; this formula was rederived in
\cite{Schu} from the axioms of root systems. It is stated as follows:%
\begin{equation}
w_{\mu}=\sum_{w\in W}\frac{e^{w(\mu)}}{\prod_{\alpha\in S}(1-e^{-w(\alpha)}%
)}\,;\label{BCF}%
\end{equation}
here, for any simple root $\alpha$ and any Weyl group element $w$, we set%
\begin{align*}
\frac{1}{1-e^{-w(\alpha)}} &  =\sum_{k=0}^{+\infty}e^{kw(\alpha)}\;\;\;\;\text{ if
}w(\alpha)\in R^{+}\,,\text{ and}\\
\frac{1}{1-e^{-w(\alpha)}} &  =-\frac{e^{w(\alpha)}}{1-e^{w(\alpha)}}%
=-\sum_{k=0}^{+\infty}e^{(k+1)w(\alpha)}\;\;\;\;\text{ if }w(\alpha)\in-R^{+}\,.
\end{align*}
Observe that $w_{\mu}^{+}(t)$ coincides with the dominant part in the expansion%
\[
\frac{e^{\mu}}{\prod_{\alpha\in S}(1-te^{-\alpha})}\,.%
\]

Based on Remark~\ref{bases}, consider the expansion
\begin{equation}
\chi_{\lambda}=\sum_{\mu\in P^{+}(\lambda)}A_{\lambda,\mu
}\,w_{\mu}\,,\;\;\;\;\;\mbox{or equivalently }\,\chi_{\lambda}^+=\sum_{\mu\in P^{+}(\lambda)}A_{\lambda,\mu
}\,w_{\mu}^{+}\,. \label{AtomDec}%
\end{equation}
Similarly, consider the polynomials $A_{\lambda,\mu}(t)$, called {\em atomic polynomials}, which are defined by
\begin{equation}\label{atdec}
\chi_{\lambda}^+(t)=\sum_{\mu\in P^{+}(\lambda)}%
A_{\lambda,\mu}(t)\,w_{\mu}^{+}(t)\,.
\end{equation}
We have $A_{\lambda,\lambda}(t)=1$. 
The following property of the atomic polynomials gives an alternative definition of them, via M\"obius inversion on the dominance order.

\begin{proposition}\label{defateq} The atomic polynomials satisfy the following relation:
\begin{equation}\label{lasc0}K_{\lambda,\nu}(t)=\sum_{\nu\le\mu\le\lambda}t^{\langle\mu-\nu,\rho^\vee\rangle}A_{\lambda,\mu}(t)\,,\;\;\;\;\;\mbox{for all $\nu\in P^{+}(\lambda)$}\,.\end{equation}
\end{proposition}

\begin{proof} The expansion \eqref{atdec} can be written
\begin{align*}\chi_{\lambda}^+(t)&=\sum_{\mu\in P^{+}(\lambda)}A_{\lambda,\mu}(t)\left(\sum_{\nu\le\mu} t^{\langle\mu-\nu,\rho^\vee\rangle}e^{\nu}\right)\\
&=\sum_{\nu\le\lambda}\left(\sum_{\nu\le\mu\le\lambda}t^{\langle\mu-\nu,\rho^\vee\rangle}A_{\lambda,\mu}(t)\right)e^\nu\,.
\end{align*}
By comparing with \eqref{domchar}, the desired relation follows.
\end{proof}

\begin{definition}\label{defat}
The character $\chi_{\lambda}$ admits an {\em atomic decomposition} if
 $A_{\lambda,\mu}\in{\mathbb Z}_{\ge 0}$ for any $\mu\in P^+(\lambda)$.
Similarly, we say that $\chi_{\lambda}^+(t)$ admits a
{\em $t$-atomic decomposition} if $A_{\lambda,\mu}%
(t)\in\mathbb{Z}_{\geq0}[t]$ for any $\mu\in P^+(\lambda)$. 
\end{definition}

By analogy with \eqref{domchar} and using \eqref{ktilde}, we define
\[\widetilde{\chi}_{\lambda}^+(t):=\sum_{\mu\in P^{+}(\lambda)}%
\widetilde{K}_{\lambda,\mu}(t)\,e^{\mu}\,.\]
Like in \eqref{atdec}, consider the polynomials $\widetilde{A}_{\lambda,\mu}(t)$ defined by
\begin{equation}\label{modatdec}\widetilde{\chi}_{\lambda}^+(t)=\sum_{\mu\in P^{+}(\lambda)}\widetilde{A}_{\lambda,\mu}(t)\,w_{\mu}^{+}\,,\end{equation}
where we recall that $w_\mu^+:=w_\mu^+(1)$.
When the context is clear, we will also refer to $\widetilde{A}_{\lambda,\mu}(t)$ as atomic polynomials. As shown below, the $\widetilde{A}_{\lambda,\mu}(t)$ are closely related to ${A}_{\lambda,\mu}(t)$, and are subject to an analogue of~\eqref{lasc0}. 

\begin{proposition}\label{defateq1} The polynomials $\widetilde{A}_{\lambda,\mu}(t)$ satisfy the following relation:
\begin{equation}\label{lasc}\widetilde{K}_{\lambda,\nu}(t)=\sum_{\nu\le\mu\le\lambda}\widetilde{A}_{\lambda,\mu}(t)\,,\;\;\;\;\;\mbox{for all $\nu\in P^{+}(\lambda)$}\,.\end{equation}
Moreover, we have
\begin{equation}\label{relat}\widetilde{A}_{\lambda,\mu}(t)=t^{\langle\lambda-\mu,\rho^\vee\rangle}A_{\lambda,\mu}(t^{-1})\,.\end{equation}
Thus, the $t$-atomic decomposition is equivalent to the fact that $\widetilde{A}_{\lambda,\mu}(t)\in\mathbb{Z}_{\geq0}[t]$. 
\end{proposition}

\begin{proof} The proof of \eqref{lasc} is completely similar to that of~\eqref{lasc0}. As for~\eqref{relat}, it can be seen in the following way:
\begin{align*}
\widetilde{K}_{\lambda,\nu}(t)&=t^{\langle\lambda-\nu,\rho^\vee\rangle}K_{\lambda,\nu}(t^{-1})=t^{\langle\lambda-\nu,\rho^\vee\rangle}\sum_{\nu\le\mu\le\lambda}t^{\langle\nu-\mu,\rho^\vee\rangle}A_{\lambda,\mu}(t^{-1})\\
&=\sum_{\nu\le\mu\le\lambda}t^{\langle\lambda-\mu,\rho^\vee\rangle}A_{\lambda,\mu}(t^{-1})=\sum_{\nu\le\mu\le\lambda}\widetilde{A}_{\lambda,\mu}(t)\,.
\end{align*}
\end{proof}

\begin{remarks}\label{lasdec} {\rm  (1) Lascoux \cite{lascpw} made a statement very closely related to~\eqref{lasc}, for type $A$. The slight difference consists in the definition of $\widetilde{K}_{\lambda,\mu}(t)$, for given partitions $\lambda,\mu$, namely $\widetilde{K}_{\lambda,\mu}(t):=t^{n(\mu)}\,K_{\lambda,\mu}(t^{-1})$, where $n(\mu):=\sum_i(i-1)\mu_i$. 
}

{\rm (2) The $t$-atomic decomposition, as stated in Proposition~\ref{defateq1}, implies the monotonicity of the Kostka-Foulkes polynomials, which holds in the full generality of Kazhdan-Lusztig polynomials for finite and affine Weyl groups \cite[Corollary~3.7]{BrM}. Indeed, the latter says that, for  $x\le y\le z$ in such a Weyl group, the difference of Kazhdan-Lusztig polynomials $P_{x,z}(t)-P_{y,z}(t)$ is in ${\mathbb Z}_{\ge 0}[t]$. For $z=w_\lambda$, $y=w_\mu$, and $x=w_\nu$, with $\nu\le\mu\le\lambda$, this follows based on \eqref{ktilde} and the fact that the atomic polynomials in the decomposition of $\widetilde{K}_{\lambda,\mu}(t)$ are among those in the decomposition of $\widetilde{K}_{\lambda,\nu}(t)$.}

{\rm (3) {We shall see} that the $t$-atomic decomposition is always true in type $A$, as mentioned in Section~\ref{Intro}. However, even the atomic decomposition (i.e., the positivity in \eqref{AtomDec}) might fail beyond type $A$, unlike it was claimed in \cite[Theorem~2.2]{Schu}. The shortest counterexample we found is in type $D_4$, and is given in Example~\ref{counterex} below. However, a slight increase in rank corrects this problem, and in fact we will see that this is a general phenomenon. 
}
\end{remarks}

\begin{example}\label{counterex}{\rm Consider $\lambda:=2\omega_1+2\omega_2$ in type $D_4$. For simplicity, we let $w_{abcd}:=w_{\mu}$ for $\mu:=a\omega_1+b\omega_2+c\omega_3+d\omega_4$. With this notation, we have:
\begin{align*}\chi_\lambda&=w_\lambda+w_{4000}+w_{1111}+w_{2002}+w_{0022}+w_{2020}+2w_{2100}+w_{0200}+\\
&+4w_{1011}+5w_{0002}+5w_{0020}+11w_{2000}-3w_{0100}+17w_{0000}\,.
\end{align*}
However, for the same $\lambda$ we obtain a positive expansion in type $D_5$. }
\end{example}

Set $t_{i}:=e^{-\alpha_{i}}\in]0,1[$, for $i=1,\ldots,r$, and
consider a sequence $(\lambda^{(k)})_{k\geq0}$ of dominant weights such that
$\underset{k\rightarrow+\infty}{\mathrm{lim}}\langle\lambda^{(k)},\alpha
_{i}\rangle=+\infty$, for any $i=1,\ldots,r$. We shall then write $\lambda
\rightarrow+\infty$ for short.

\begin{proposition}
\label{prop_limit}Under the previous assumption we have%
\[
\underset{\lambda\rightarrow+\infty}{\mathrm{lim}}e^{-\lambda}\,\chi_{\lambda
}=\prod_{\alpha\in R^{+}}\frac{1}{1-e^{-\alpha}}\;\;\;\;\text{ and }\;\;\;\;\underset
{\lambda\rightarrow+\infty}{\mathrm{lim}}e^{-\lambda}\,w_{\lambda}=\prod
_{\alpha\in S}\frac{1}{1-e^{-\alpha}}\,.
\]
\end{proposition}

\begin{proof}
We have by the Weyl character formula%
\[
e^{-\lambda^{(k)}}\chi_{\lambda^{(k)}}=\frac{\sum_{w\in W}(-1)^{\ell
(w)}e^{w(\lambda^{(k)}+\rho)-\rho-\lambda^{(k)}}}{\prod_{\alpha\in R^{+}%
}(1-e^{-\alpha})}\,.
\]
Let $t=\mathrm{max}(t_{1},\ldots,t_{r})\in]0,1[$. Since $w(\lambda^{(k)}%
+\rho)-\rho-\lambda^{(k)}\in-Q^{+}$, we can set%
\[
w(\lambda^{(k)}+\rho)-\rho-\lambda^{(k)}=-\sum_{i=1}^{r}a_{i}(\lambda
^{(k)})\,\alpha_{i}\,,%
\]
where $a_{i}(\lambda^{(k)})\in\mathbb{Z}_{\geq0}$ for any $i=1,\ldots,r$. The
hypothesis $\underset{k\rightarrow+\infty}{\mathrm{lim}}\langle\lambda^{(k)}%
,\alpha_{i}\rangle=+\infty$ for any $i=1,\ldots,r$ implies that $\underset
{k\rightarrow+\infty}{\mathrm{lim}}\sum_{i=1}^{r}a_{i}(\lambda^{(k)})=+\infty$
for any $w\in W$ such that $w\neq1$. So we get for any such $w$%
\[
\underset{k\rightarrow+\infty}{\mathrm{lim}}e^{w(\lambda^{(k)}+\rho
)-\rho-\lambda^{(k)}}\leq\underset{k\rightarrow+\infty}{\mathrm{lim}}%
t^{\sum_{i=1}^{r}a_{i}(\lambda^{(k)})}=0\,.
\]
Since for $w=1$ we have $(-1)^{\ell(w)}e^{w(\lambda^{(k)}+\rho)-\rho
-\lambda^{(k)}}=1$, we get our first limit.

From (\ref{BCF}), we can write
\[
e^{-\lambda^{(k)}}w_{\lambda^{(k)}}=\sum_{w\in W}\frac{e^{w(\lambda
^{(k)})-\lambda^{(k)}}}{\prod_{\alpha\in S}(1-e^{-w(\lambda)})}%
\]
and by using similar arguments, only $w=1$ contributes when we consider the limit.
\end{proof}

Now set%
\[
\prod_{\alpha\in R^{+}\setminus S}\frac{1}{1-te^{-\alpha}}=\sum_{\beta\in
Q^{+}}\emph{M}_{t}(\beta)\,e^{-\beta}\,.
\]
In particular, $\emph{M}_{1}(\beta)$ is the number of decompositions of
$\beta$ as a sum of nonsimple positive roots. We get%
\begin{equation}
\prod_{\alpha\in R^{+}}\frac{1}{1-te^{-\alpha}}=\sum_{\beta\in Q^{+}}%
\emph{M}_{t}(\beta)\prod_{\alpha\in S}\frac{e^{-\beta}}{1-te^{-\alpha}}\,.
\label{decB(infinity)}%
\end{equation}
Proposition~\ref{prop_limit} and (\ref{BCF}) suggest to consider (\ref{decB(infinity)}) as a
$t$-analogue of (\ref{AtomDec}) when $\lambda\rightarrow+\infty$.

\subsection{Atomic decomposition of finite crystals}\label{sec:tatdec}

Let $B(\lambda)^{+}$ be the subset of $B(\lambda)$ of vertices with dominant weights.

\begin{definition}\label{crat}
An {\em atomic decomposition} of the crystal $B(\lambda)$ is a partition
\begin{equation}
B(\lambda)^{+}=%
{\bigsqcup_{h\in H(\lambda)}}
\mathbb{B}(\lambda,h)\,, \label{part}%
\end{equation}
where $H(\lambda)\subset B(\lambda)^{+}$, $h\in\mathbb{B}(\lambda,h)$ is a distinguished vertex, and each component $\mathbb{B}(\lambda,h)$ consists of exactly one vertex
of dominant weight $\nu$ for each $\nu\leq\mathrm{wt}(h)$.
\end{definition}

Observe that the cardinality of $\mathbb{B}(\lambda,h)$ is then independent of
$\lambda$, and if $\mathrm{wt}(h)=\mu$ we have
\begin{equation}\label{wmucomb}
w_{\mu}^{+}=\sum_{b\in\mathbb{B}(\lambda,h)}e^{\mathrm{wt}(b)}\,.
\end{equation}
If $B(\lambda)$ has an atomic decomposition, then clearly $\chi_{\lambda}^+$ has the atomic decomposition \eqref{AtomDec}, %
where $A_{\lambda,\mu}$ is the number of vertices of weight $\mu$ in
$H(\lambda).$

\begin{definition}\label{crtat}
A {\em $t$-atomic decomposition} of the crystal $B(\lambda)$ is an atomic decomposition together with a statistic $\mathrm{c}%
:H(\lambda)\rightarrow\mathbb{Z}_{\geq0}$ such that the following polynomials satisfy {\rm \eqref{atdec}}:
\begin{equation}\label{combat}
A_{\lambda,\mu}(t)=\sum_{\substack{h\in H(\lambda) \\ \mathrm{wt}(h)=\mu}}t^{\mathrm{c}%
(h)}\,.%
\end{equation}
\end{definition}

As we can see, the $t$-atomic 
decomposition property of $\chi_{\lambda}^+(t)$ is part of Definition~\ref{crtat}. Assuming that $B(\lambda)$ has a $t$-atomic decomposition, one can extend the statistic $\mathrm{c}$ to $B(\lambda)^+$
by setting 
\begin{equation}\label{propac}\mathrm{c}(b):=\mathrm{c}(h)+\langle\mathrm{wt}(h)-\mathrm{wt}(b),\rho^\vee\rangle\,,\;\;\;\;\;\mbox{for any $b\in\mathbb{B}(\lambda,h)$}\,.
\end{equation}
The $t$-analogues of the combinatorial formulas \eqref{wmucomb} and \eqref{domchar} immediately follow from Definition~\ref{crtat}:
\begin{align}
w_{\mu}^{+}(t)&=\sum_{b\in\mathbb{B}(\lambda,h)}t^{\mathrm{c}(b)-\mathrm{c}(h)}e^{\mathrm{wt}(b)}\,,\\
\chi_{\lambda}^+(t)&=\sum_{b\in B(\lambda)^+}t^{\mathrm{c}(b)}e^{\mathrm{wt}(b)}\,.\label{tcombcr}
\end{align}
Moreover, by comparing \eqref{tcombcr} with \eqref{domchar}, we obtain the following combinatorial formula for Kostka-Foulkes polynomials:
\begin{equation}\label{combkf}
K_{\lambda,\mu}(t)=\sum_{\substack{b\in B(\lambda) \\ \mathrm{wt}(b)=\mu}} t^{\mathrm{c}(b)}\,.
\end{equation}

To summarize, the existence of a $t$-atomic decomposition of a crystal is highly desirable because: (i) it implies the $t$-atomic decomposition of $\chi_{\lambda}^+(t)$ and of the Kostka-Foulkes polynomials $K_{\lambda,\mu}(t)$, which are now realized combinatorially; (ii) it leads to combinatorial formulas for both $K_{\lambda,\mu}(t)$ and the atomic polynomials $A_{\lambda,\mu}(t)$, namely \eqref{combkf} and \eqref{combat}, respectively. 


\section{Atomic decomposition of the crystal $B(\infty)$}\label{sec:inf}

In this section we assume that the Lie algebra $\mathfrak{g}$ is of type
$A_{n-1}$, $B_{n}$, $C_{n}$, $D_n$, and $G_{2}$. Let $r$ be the rank of $\mathfrak{g}$.

\subsection{Marginally large tableaux}

{\em Marginally large tableaux} were introduced by Hong and Lee \cite{HL1} in order to describe the
crystal $B(\infty)$ associated to $\mathfrak{g}$. 
Recall that they can be regarded as
$\mathfrak{g}$-tableaux (that is, of type $A_{n-1}$, $B_{n}$, $C_{n}$, $D_n$, or $G_{2}$) with $d$ rows such that%
\[
d=n\text{ in types }B_{n},\,C_{n}\,,\quad\quad d=n-1\text{ in types }A_{n-1},\,D_{n}\,,%
\quad\quad\text{and }d=2\text{ in type }G_{2}\,;%
\]
furthermore, for any $i=1,\ldots,d-1$, the number of boxes in row $i$ containing $i$ is
equal to $1$ plus the number of boxes in row $i+1$ (see the example below).
Write $\mathcal{T}(\infty)$ for the set of marginally large tableaux
associated to $\mathfrak{g}$.\ Since marginally large tableaux are special cases of 
tableaux for each type considered, the set $\mathcal{T}(\infty)$ comes with a
crystal action which is essentially the same as in the finite crystal up to
renormalization of rows, in order to insure that the obtained tableau is marginally
large. This renormalization is defined as follows. Consider a marginally large
$\mathfrak{g}$-tableau $T$, and let $T'$ be a $\mathfrak{g}$-tableau
obtained from $T$ by modifying a letter in row $i$. If $T'$ is not
marginally large, this means that we have modified the rightmost letter $i$ in row
$i$. Then the renormalization of $\widehat{T'}$ is the marginally
large tableau obtained from $T'$ by adding a letter $k$ in each row
$k$ between $1$ and $i$, the others rows remaining unchanged. One can then
define crystal operators $\tilde{F}_{1},\ldots,\tilde{F}_{r}$ on
$\mathcal{T}(\infty)$ by setting%
\[
\tilde{F}_{i}(T)=\widehat{\tilde{f}_{i}(T)}\,,%
\]
where $\tilde{f}_{i}$ is the ordinary Kashiwara crystal operator on the
$\mathfrak{g}$-tableau $T$. It was established in \cite{HL1} that the crystal
structure on $\mathcal{T}(\infty)$ obtained in this way is isomorphic to $B(\infty)$.

\begin{example}{\rm 
If $\mathfrak{g}$ is of type $A_{3}$, then
\[
T=%
\begin{tabular}
[c]{|l|l|llllllll}\hline
$1$ & $1$ & $1$ & \multicolumn{1}{|l}{$1$} & \multicolumn{1}{|l}{$1$} &
\multicolumn{1}{|l}{$1$} & \multicolumn{1}{|l}{$1$} & \multicolumn{1}{|l}{$3$}
& \multicolumn{1}{|l}{$5$} & \multicolumn{1}{|l|}{$5$}\\\hline
$2$ & $2$ & $2$ & \multicolumn{1}{|l}{$3$} & \multicolumn{1}{|l}{$3$} &
\multicolumn{1}{|l}{$4$} & \multicolumn{1}{|l}{} &  &  & \\\cline{1-6}%
$3$ & $5$ &  &  &  &  &  &  &  & \\\cline{1-2}%
\end{tabular}
\]
is a marginally large tableau, and
\[
\tilde{f}_{2}(T)=%
\begin{tabular}
[c]{|l|l|lllllllll}\hline
$1$ & $1$ & $1$ & \multicolumn{1}{|l}{$1$} & \multicolumn{1}{|l}{$1$} &
\multicolumn{1}{|l}{$1$} & \multicolumn{1}{|l}{$1$} & \multicolumn{1}{|l}{$1$}
& \multicolumn{1}{|l}{$3$} & \multicolumn{1}{|l}{$5$} &
\multicolumn{1}{|l|}{$5$}\\\hline
$2$ & $2$ & $2$ & \multicolumn{1}{|l}{$3$} & \multicolumn{1}{|l}{$3$} &
\multicolumn{1}{|l}{$3$} & \multicolumn{1}{|l}{$4$} & \multicolumn{1}{|l}{} &
&  & \\\cline{1-2}\cline{1-7}%
$3$ & $5$ &  &  &  &  &  &  &  &  & \\\cline{1-2}%
\end{tabular}
\,.
\]}
\end{example}

We define a multisegment as a multiset of positive roots.\ We can write a
multisegment $\mathfrak{m}$ as%
\[
\mathfrak{m}=\sum_{\alpha\in R^{+}}m_{\alpha}\,\alpha\,,
\]
which means that the multiset $\mathfrak{m}$ contains $m_{\alpha}$ times the
positive root $\alpha$. Let $\mathfrak{M}$ be the set of $\mathfrak{g}%
$-multisegments. In \cite[Proposition~3.7]{LeSa}, a bijection $\Xi
$\footnote{As far as we are aware, such a bijection is not known in types
$E_{6}$, $E_{7}$, $E_{8}$, and $F_{4}$.} is given from $\mathcal{T}(\infty)$ to $\mathfrak{M}%
$. This bijection depends on the type considered. For example, in type
$A_{n-1}$, for any positive root $\alpha_{ij}=\varepsilon_{i}-\varepsilon
_{j}$ with $1\leq i<j\leq n$ and any marginally large tableau $T$, the integer
$m_{\alpha_{i,j}}$ is equal to the number of letters $j$ in the $i$-th row of
$T$. For any $T\in\mathcal{T}(\infty)$ such that $\Xi(T)=\sum_{\alpha\in
R^{+}}m_{\alpha}\,\alpha$, write $\left\vert T\right\vert =\sum_{\alpha\in
R^{+}}m_{\alpha}$.

\begin{example}{\rm 
Resuming the previous example, we get%
\[
\Xi(T)=\alpha_{1,3}+2\alpha_{1,5}+2\alpha_{2,3}+\alpha_{2,4}+\alpha_{3,5}%
\]
and $\left\vert T\right\vert =7$.}
\end{example}

\subsection{Modified crystal operators and atomic decomposition of $B(\infty
)$}

We refer to \cite{LeSa} for a complete description of the bijection $\Xi$ in
each case. In the sequel, we only need the following two properties of
the map $\Xi$.

\begin{enumerate}
\item Starting from any marginally large tableau $T$ and any simple root
$\alpha_{i},i=1,\ldots,r$, there is a unique marginally large tableau
$T'$ such that
\[
\Xi(T')=\Xi(T)+\alpha_{i}\,.
\]
We then set $T'=\mathrm{F}_{i}(T)$.

\item For any marginally large tableau $T'$ and any $i=1,\ldots,n$
such that $\Xi(T')=\sum_{\alpha\in R^{+}}m_{\alpha}\,\alpha$ with
$m_{\alpha_{i}}>0$, there exists a (unique) marginally large tableau $T$ such
that $T'=\mathrm{F}_{i}(T)$.
\end{enumerate}

More precisely, the construction of $T'$ from $T$ is as follows.

\begin{itemize}
\item For any $i=1,\ldots,r-1$, $T'$ is obtained from $T$ by replacing
the rightmost letter $i$ located in row $i$ by a letter $i+1$ and then by
renormalizing if needed.

\item For $i=2$ in type $G_{2}$, $T'$ is obtained from $T$ by
replacing the rightmost letter $2$ located in row $2$ by a letter $3$.

\item For $i=n$ in type $C_{n}$, $T'$ is obtained from $T$ by
replacing the rightmost letter $n$ located in row $n$ by a letter
$\overline{n}$.

\item For $i=n$ in type $D_{n}$, $T'$ is obtained from $T$ by
replacing the rightmost letter $n-1$ located in row $n-1$ by a letter
$\overline{n}$.

\item For $i=n$ in type $B_{n}$, the rows of $T$ can only contain one letter
$0$. Then we have:

\begin{itemize}
\item[-] if $0$ belongs to row $n$, $T'$ is obtained from $T$ by
replacing this letter $0$ by a letter $\overline{n}$;

\item[-] if $0$ does not belong to row $n$, $T'$ is obtained from $T$
by replacing the rightmost letter $n$ located in row $n$ by a letter $0$.
\end{itemize}
\end{itemize}

Observe that $\mathrm{F}_{i}(T)$ is a marginally large tableau for any
$i=1,\ldots,n$ and any $T\in\mathcal{T}(\infty)$. We also define the operators
$\mathrm{E}_{i}$, $i=1,\ldots,n$, such that $\mathrm{E}_{i}(T_{1})=T_{2}$ if there
exists $T_{2}\in\mathcal{T}(\infty)$ satisfying $\mathrm{F}_{i}(T_{2})=T_{1}$,
and $\mathrm{E}_{i}(T_{1})=0$ otherwise. The operators $\mathrm{F}_{i}$ and
$\mathrm{E}_{i}$, $i=1,\ldots,n$, are the modified crystal operators. We have
$\mathrm{F}_{i}(T)\neq\tilde{F}_{i}(T)$ and $\mathrm{E}_{i}(T)\neq\tilde
{F}_{i}(T)$ in general. Furthermore, by property (2), we have $\mathrm{E}_{i}(T)\neq0$
for any $T$ such that $m_{\alpha_{i}}>0$. Now we can endow $\mathcal{T}%
(\infty)$ with a new colored directed graph structure $\mathbb{B}(\infty)$ such
that $T\overset{i}{\dashrightarrow}T'$ if and only if $T^{\prime
}=\mathrm{F}_{i}(T)$. A source vertex for this structure is a marginally large
tableau $T$ such that $\mathrm{E}_{i}(T)=0$ for any $i=1,\ldots,r$. Let us
denote by $\mathcal{ST}(\infty)$ the set of source vertices in $\mathcal{T}%
(\infty)$.

\begin{theorem}\hfill
\begin{enumerate}
\item We have $S\in\mathcal{ST}(\infty)$ if and only if $\Xi(T)=\sum
_{\alpha\in R^{+}\setminus S}m_{\alpha}\,\alpha,$ that is, when $m_{\alpha_{i}%
}=0$ for any $i=1,\ldots,r$.

\item Each connected component $\mathbb{B}$ of the graph $\mathbb{B}(\infty)$
contains a unique source vertex $S$. We then write $\mathbb{B}=\mathbb{B}(S)$,
and say that $\mathbb{B}(S)$ is an atom of $B(\infty)$.

\item For any $S\in\mathcal{ST}(\infty)$, the vertices of $\mathbb{B}(S)$ have
different weights, and
\begin{equation}
\sum_{T\in\mathbb{B}(S)}t^{\left\vert T\right\vert }e^{\mathrm{wt}(T)}%
=\frac{t^{\left\vert S\right\vert }e^{\mathrm{wt}(S)}}{\prod_{\alpha\in
S}(1-te^{-\alpha})}\,.\label{charB(S)}%
\end{equation}
\end{enumerate}
\end{theorem}

\begin{proof}
The first part follows from the fact that
\[
\Xi(\mathrm{E}_{i}(T))=\Xi(T)-\alpha_{i}\,,%
\]
for any $i=1,\ldots,r$. For the second part, we observe that, for any
marginally large tableau $T$ and for any $i\neq j$ in $\{1,\ldots,r\}$, we have
$\mathrm{E}_{i}\mathrm{E}_{j}(T)\neq0$ and $\mathrm{E}_{i}\mathrm{E}%
_{j}(T)\neq0$ if and only if $m_{\alpha_{i}}>0$ and $m_{\alpha_{j}}>0$; in this case, we have
 $\mathrm{E}_{i}\mathrm{E}_{j}(T)=\mathrm{E}_{i}\mathrm{E}_{j}(T)$.
Thus
\[
S=\prod_{i=1}^{r}\mathrm{E}_{i}^{m_{\alpha_{i}}}(T)
\]
is the unique source vertex of the connected component corresponding to $T$,
and does not depend on the order in which the operators $\mathrm{E}_{i}$ are
applied in the right hand side. For the third part, observe we have
\[
\mathbb{B}(S)=\left\{  \prod_{i=1}^{r}\mathrm{F}_{i}^{m_{\alpha_{i}}}(S)\mid
m_{\alpha_{i}}\in\mathbb{Z}_{\geq0}\right\}\,,
\]
where the operators $\mathrm{F}_{i},$ $i=1,\ldots,r$ commute.\ Since we have
\[
\mathrm{wt}\left(  \prod_{i=1}^{r}\mathrm{F}_{i}^{m_{\alpha_{i}}}(S)\right)
=\mathrm{wt}(S)-\sum_{i=1}^{r}m_{\alpha_{i}}\alpha_{i}\,,%
\]
the weights of the vertices in $\mathbb{B}(S)$ are all distinct. This also
yields the desired equality (\ref{charB(S)}).
\end{proof}

\begin{corollary}
The partition
\[
\mathbb{B}(\infty)=%
{\bigsqcup_{S\in\mathcal{ST}(\infty)}}
\mathbb{B}(S)
\]
is a $t$-atomic decomposition of $B(\infty);$ in other words, we have
\[
\prod_{\alpha\in R^{+}}\frac{1}{1-te^{-\alpha}}=\sum_{\beta\in Q^{+}}%
{M}_{t}(\beta)\prod_{\alpha\in S}\frac{e^{-\beta}}{1-te^{-\alpha}}\,,%
\]
where
\[
{M}_{t}(\beta)=\sum_{\substack{S\in\mathcal{ST}(\infty) \\ \mathrm{wt}(S)=\beta
}}t^{\left\vert S\right\vert }\,.
\]
\end{corollary}

\section{The partial order on dominant weights}\label{sec:pod}
Before we consider the atomic decomposition of finite crystals, we need some information about the partial order on dominant weights that was defined in Section~\ref{sec:backgr}. In full generality, this poset was first studied in \cite{stepod}, so we will recall some results from this paper. 

The components of the dominant weight poset are lattices. Each cocover is of the form $\mu\gtrdot\mu-\alpha$, where $\alpha$ is a positive root, so we can represent it as a downward edge in the Hasse diagram labeled by $\alpha$. The cocovers were completely described in \cite[Theorem~2.8]{stepod}. Fixing a dominant weight $\lambda$, we will be interested in the lower order ideal determined by $\lambda$. This is known to be an interval $[\widehat{0},\lambda]$, with $\widehat{0}$ a minimal element of the dominant weight poset. 

\subsection{Type $A_{n-1}$} \label{dwpa}

Now $\lambda$ is a partition $(\lambda_1\ge\ldots\ge\lambda_{n-1}\ge 0)$, and let $N=|\lambda|:=\sum_i\lambda_i$. We identify partitions with their Young diagrams, and we denote a partition with $p$ parts $a$, $q$ parts $b$ ($a\ge b$) etc. by $(a^pb^q\ldots)$. As explained below, the interval $[\widehat{0},\lambda]$ mentioned above can be identified with the similar interval in the poset of partitions of $N$ with the {\em dominance order}; the latter is defined by $(\mu_1,\mu_2,\ldots)\le(\nu_1,\nu_2,\ldots)$ if and only if $\sum_{i=1}^j\mu_i\le\sum_{i=1}^j\nu_j$, for any $j$. As a partition of $N$, the element $\widehat{0}$ is the partition with $\lfloor N/n\rfloor$ columns of height $n$, and the last column of height $p:=N\;\mbox{mod}\;n$; as an element of the dominant weight poset, $\widehat{0}$ is the fundamental weight $\omega_p$. We identify a partition of $N$ with the partition obtained from it by removing all columns of height $n$ (if a partition of $N$ is greater or equal to $\widehat{0}$, then it has no columns of height larger than $n$). 

The dominance order on partitions of $N$ was first studied in \cite{brylip}, so we recall some results in this paper. Conjugation of partitions is an antiautomorphism. The cocovers $\mu\gtrdot\mu-\alpha_{ij}$, where $\alpha_{ij}=\varepsilon_i-\varepsilon_j$ is a positive root (i.e., $i<j$), are labeled by $(i,j)$. It turns out that there are only two types of cocovers, namely:
\begin{equation}\label{cocovera}
(\ldots ab \ldots)\gtrdot(\ldots(a-1)(b+1)\ldots)\,,\;\;\;\;\;\mbox{and}\;\;\;\;\;(\ldots (a+1)a^p(a-1)\ldots)\gtrdot(\ldots a^{p+2}\ldots)\,,
\end{equation}
where in the first case $a\ge b+2$ and the cocover is labeled by a simple root. These types are referred to as (*) and (**), respectively, while a cocover of type (**) which is not of type (*) is called {\em proper}. 

An important result in \cite{brylip} concerns the structure of short intervals in the dominance order. To state it, we need some more definitions. Consider two distinct cocovers $\mu\gtrdot\nu$ and $\mu\gtrdot\pi$ of a partition $\mu$, which are labeled $(i,j)$ and $(k,l)$, where we assume $i<k$. These cocovers can only have one of the following relative positions (in terms of their labels): (i) {\em nonoverlapping} if $j<k$; (ii) {\em partially overlapping} if $j=k$; (iii) {\em fully overlapping} if $k=j-1$. By \cite[Proposition~3.2]{brylip}, the interval $[\nu\wedge\pi,\mu]$ can only have one of the following structures; the two cocovers above are shown in the diagrams below in bold.

{\bf Case A1:} cocovers which are (a) nonoverlapping; (b) partially overlapping and both of type (*); (c) fully overlapping and both proper of type (**). As subcase (a) is easy, only subcases (b) and (c) are represented in the diagrams below. 

In subcase (b), we have $a\ge c+2$ and $c\ge e+2$, while $i$ is the position of $a$ in the partition $\mu$. 
\[\scriptstyle
{\xymatrix{
&{\ldots ace\ldots}\ar[dl]_{(i,i+1)}\ar[dr]^{(i+1,i+2)}\ar@<0.2mm>[dl]\ar@<0.4mm>[dl]\ar@<0.2mm>[dr]\ar@<0.4mm>[dr]\\
{\ldots (a-1)(c+1)e\ldots}\ar[dr]_{(i+1,i+2)}&&{\ldots a(c-1)(e+1)\ldots}\ar[dl]^{(i,i+1)}\\
&{\ldots (a-1)c(e+1)\ldots}
} }
\]

In subcase (c), we have $b=a-1$, $c=b-1$, $d=c-1$, $p,q\ge 1$, while $i$ is the position of $a$, $j=i+p+1$ is the position of the first $c$, and $k=j+q$ is the position of $d$.
\[\scriptstyle
{\xymatrix{
&{\ldots ab^pc^qd\ldots}\ar[dl]_{(i,j)}\ar[dr]^{(j-1,k)}\ar@<0.2mm>[dl]\ar@<0.4mm>[dl]\ar@<0.2mm>[dr]\ar@<0.4mm>[dr]\\
{\ldots b^{p+2}c^{q-1}d\ldots}\ar[dr]_{(j,k)}&&{\ldots ab^{p-1}c^{q+2}\ldots}\ar[dl]^{(i,j-1)}\\
&{\ldots b^{p+1}c^{q+1}\ldots}
} }
\]

{\bf Case A2:} partially overlapping cocovers, where (a) the first is of type (*) and the second proper of type (**); (b) vice versa. 

In subcase (a), we have $a\ge c+2$, $d=c-1$, $e=d-1$, $p\ge 1$, while $i$ is the position of $a$ in the partition $\mu$ and $j=i+p+2$ is the position of $e$. 
\[\scriptstyle
{\xymatrix{
&{\ldots acd^pe\ldots}\ar[dl]_{(i,i+1)}\ar[ddr]^{(i+1,j)}\ar@<0.2mm>[dl]\ar@<0.4mm>[dl]\ar@<0.2mm>[ddr]\ar@<0.4mm>[ddr]\\
{\ldots(a-1)(c+1)d^pe\ldots}\ar[dd]_{(i+1,i+2)}\\
&&{\ldots ad^{p+2}\ldots}\ar[ddl]^{(i,i+1)}\\
{\ldots(a-1)c^2d^{p-1}e\ldots}\ar[dr]_{(i+2,j)}\\
&{\ldots(a-1)cd^{p+1}\ldots}} }
\]

In subcase (b), we have $b=a-1$, $c=b-1$, $c\ge e+2$, $p\ge 1$, while $i$ is the position of $a$ and $j=i+p+1$ is the position of $c$. 
\[\scriptstyle
{\xymatrix{
&{\ldots ab^pce\ldots}\ar[ddl]_{(i,j)}\ar[dr]^{(j,j+1)}\ar@<0.2mm>[ddl]\ar@<0.4mm>[ddl]\ar@<0.2mm>[dr]\ar@<0.4mm>[dr]\\
&&{\ldots ab^p(c-1)(e+1)\ldots}\ar[dd]^{(j-1,j)}\\
{\ldots b^{p+2}e\ldots}\ar[ddr]_{(j,j+1)}\\
&&{\ldots ab^{p-1}c^2(e+1)\ldots}\ar[dl]^{(i,j-1)}\\
&{\ldots b^{p+1}c(e+1)\ldots}} }
\]

{\bf Case A3:} partially overlapping cocovers, both proper of type (**). Here $b=a-1$, $c=b-1$, $d=c-1$, $e=d-1$, $p,q\ge 1$, while $i$ is the position of $a$ in the partition $\mu$, $j=i+p+1$ is the position of $c$, and $k=j+q+1$ is the position of $e$. 
\[\scriptstyle
{\xymatrix{
&{\ldots ab^pcd^qe\ldots } \ar[d]_{(j-1,j+1)}\ar[dl]_{(i,j)}\ar[dr]^{(j,k)}\ar@<0.2mm>[dl]\ar@<0.4mm>[dl]\ar@<0.2mm>[dr]\ar@<0.4mm>[dr]    \\
{\ldots b^{p+2}d^qe\ldots } \ar[d]_{(j,j+1)} &{\ldots {ab^{p-1}c^3d^{q-1}e}\ldots } \ar[dl]^{(i,j-1)} \ar[dr]_{(j+1,k)} &{\ldots ab^pd^{q+2}\ldots} \ar[d]^{(j-1,j)}\\
{\ldots b^{p+1}c^2d^{q-1}e\ldots } \ar[dr]_{(j+1,k)}&& {\ldots ab^{p-1}c^2d^{q+1}\ldots} \ar[dl]^{(i,j-1)}\\
&{\ldots b^{p+1}cd^{q+1}\ldots }
} }
\]

Due to the conjugation automorphism of the dominance order, given two distinct covers $\mu\lessdot\nu$ and $\mu\lessdot\pi$ of $\mu$, the isomorphism type of the interval $[\mu,\nu\vee\pi]$ is always given by one of the above graphs turned upside down. Observe nevertheless the type (*) or (**) is not preserved by conjugation in general. So we have the corresponding Cases A$1'-$A$3'$. In fact, the structure of the intervals in Cases A$1'$ and A$2'$ is identical with that in Cases A1 and A2, respectively; however, Case A3 leads to the different structure shown below; the two covers above are again shown in bold. 

{\bf Case A3$\mathbf{'}$.} Here $a\ge c+1$, $d=c-1$, $d\ge f+1$, while $i$ is the position of $a$ in the partition $\mu$ at the bottom. 
\[\scriptstyle
{\xymatrix{
&{\ldots (a+1)cd(f-1)\ldots }  \ar[dl]_{(i+2,i+3)} \ar[dr]^{(i,i+1)} \\
{\ldots (a+1)c(d-1)f \ldots } \ar[d]_{(i+1,i+2)}\ar[dr]^{(i,i+1)}&&{\ldots a(c+1)d(f-1)\ldots} \ar[dl]_{(i+2,i+3)}\ar[d]^{(i+1,i+2)}  \\
{\ldots (a+1)d^2f\ldots }\ar[dr]_{(i,i+1)}\ar@<0.2mm>[dr]\ar@<0.4mm>[dr] &  {\ldots a(c+1)(d-1)f \ldots}\ar[d]_{(i+1,i+2)} & {\ldots ac^2(f-1)\ldots}\ar[dl]^{(i+2,i+3)}\ar@<0.2mm>[dl]\ar@<0.4mm>[dl]\\
&{\ldots acdf\ldots }
} }
\]

\subsection{Type $C_n$}\label{domwc}

Now $\lambda$ is a partition $(\lambda_1\ge\ldots\ge\lambda_{n}\ge 0)$, and we use the same notation as in Section~\ref{dwpa}. It is easy to see that the minimal element $\widehat{0}$ mentioned above (i.e., the unique minimal element below $\lambda$) is either $0$ or $\omega_1=(10^{n-1})$, depending on $|\lambda|$ being even or odd, respectively.
By \cite[Theorem~2.8]{stepod}, a cocover in the corresponding partial order on dominant weights is either of the same form \eqref{cocovera} as in type $A$, or has one of the following three forms:
\begin{equation}\label{cocoverc}
(\ldots\, 1^20^{n-k-1})\gtrdot(\ldots\, 0^{n-k+1})\,,\;\;\;\;\;(\ldots\, 21)\gtrdot(\ldots\, 10)\,,\;\;\;\;\;(\ldots\,(a+2))\gtrdot(\ldots\,a)\,,
\end{equation}
where $1\le k\le n-1$; these three cocovers are labeled by the roots $\varepsilon_k+\varepsilon_{k+1}$, $\varepsilon_{n-1}+\varepsilon_n$, and $2\varepsilon_n$, respectively. For simplicity, we denote a root $\varepsilon_i+\varepsilon_j$ by $\alpha_{i\overline{\jmath}}$ or $(i,\overline{\jmath})$. 
In the sequel, we will see that it is desirable for the last cover in \eqref{cocoverc} never to appear; the reason is that $\alpha_n$ is a long root, and hence does not appear in the $W$-orbit of $\alpha_1$. In fact, there is an easy condition which simplifies the setup even more.

\begin{proposition}\label{stable-crit}
If $n>(|\lambda|+1)/2$, then the first cocover in {\rm \eqref{cocoverc}} is the only one which can appear in the Hasse diagram of the interval $[\widehat{0},\lambda]$ beside the type $A$ cocovers in {\rm \eqref{cocovera}}.
\end{proposition}

\begin{proof}
Assume that a partition $\mu\le\lambda$ has the form $(\ldots\,21)$ or $(\ldots\,2)$. Then we must have $\mu_1=\ldots=\mu_{n-1}\ge 2$. Combining this with the fact that $|\lambda|\ge|\mu|$, we obtain $|\lambda|\ge 2n-1$ or $|\lambda|\ge 2n$, respectively. But this contradicts the condition in the proposition. 
\end{proof}

From now on, we work under the assumption of Proposition~\ref{stable-crit}, and we call this the {\em type $C$ stable range}. 


To the authors' knowledge, an analogue of the classification of short intervals that was described above in type $A$ is not available beyond type $A$. In order to address this problem in type $C_n$, in the stable range, we focus on the new cases involving a pair of cocovers and covers. 

Given distinct cocovers $\mu\gtrdot\nu$ and $\mu\gtrdot\pi$ of a partition $\mu$, we can assume that they are labeled $(i,j)$ and $(k,\overline{k+1})$, where necessarily $i<k$. It is easy to see that these cocovers can only have one of the following relative positions (in terms of their labels): (i) nonoverlapping if $j<k$; (ii) partially overlapping if $j=k$; (iii) fully overlapping if $j=k+2$. We are led to Cases C1$-$C3 below; the two cocovers above are shown in the diagrams below in bold, and for simplicity we omit the trailing $0$'s in a partition. Thus, we proved the following result.

\begin{proposition}\label{intcdn} Under the previous assumptions on $\nu$ and $\pi$, a lower bound of $\nu$ and $\pi$ is always obtained as in one of the Cases {\rm C1$-$C3}.
\end{proposition}

Furthermore, we claim that the structure of the interval $[\nu\wedge\pi,\mu]$ is always given by one of the diagrams below. The proof is completely similar to that of \cite[Proposition~3.2]{brylip}, which was mentioned above. However, Proposition~\ref{intcdn} suffices for our purposes.

{\bf Case C1:} cocovers which are (a) nonoverlapping; (b) fully overlapping, with $(i,j)$ proper of type (**). As subcase (a) is easy, only subcase (b) is represented in the diagram below. 

In subcase (b), we have $p\ge 2$, while $i$ is the position of (the shown) $2$ in the partition $\mu$, and $j=i+p+1$ is the position of the first $0$. 
\[\scriptstyle
{\xymatrix{
&{\ldots 21^p}\ar[dl]_{(i,j)}\ar[dr]^{(j-2,\overline{j-1})}\ar@<0.2mm>[dl]\ar@<0.4mm>[dl]\ar@<0.2mm>[dr]\ar@<0.4mm>[dr]\\
{\ldots 1^{p+2}}\ar[dr]_{(j-1,\overline{\jmath})}&&{\ldots 21^{p-2}}\ar[dl]^{(i,j-2)}\\
&{\ldots 1^p}
} }
\]

{\bf Case C2:} partially overlapping cocovers, with $(i,j)$ of type (*), so $j=i+1$. We have $a\ge 3$, while $i$ is the position of $a$ in the partition $\mu$. 
\[\scriptstyle
{\xymatrix{
&{\ldots a1^2}\ar[dl]_{(i,i+1)}\ar[ddr]^{(i+1,\overline{i+2})}\ar@<0.2mm>[dl]\ar@<0.4mm>[dl]\ar@<0.2mm>[ddr]\ar@<0.4mm>[ddr]\\
{\ldots(a-1)21}\ar[dd]_{(i+1,i+3)}\\
&&{\ldots a}\ar[ddl]^{(i,i+1)}\\
{\ldots(a-1)1^3}\ar[dr]_{(i+2,\overline{i+3})}\\
&{\ldots(a-1)1}} }
\]

{\bf Case C3:} partially overlapping cocovers, with $(i,j)$ proper of type (**). Here $p\ge 1$, while $i$ is the position of (the shown) $3$ in the partition $\mu$,  and $j=i+p+1$ is the position of the first $1$. 
\[\scriptstyle
{\xymatrix{
&{\ldots 32^p1^2 } \ar[d]_{(j-1,j+2)}\ar[dl]_{(i,j)}\ar[dr]^{(j,\overline{j+1})}\ar@<0.2mm>[dl]\ar@<0.4mm>[dl]\ar@<0.2mm>[dr]\ar@<0.4mm>[dr] \\
{\;\;\;\;\;\ldots 2^{p+2}1 \;\;\;\;\;} \ar[d]_{(j,j+2)} &{\;\;\;\;\;\ldots 32^{p-1}1^4 \;\;\;\;\;} \ar[dl]^{(i,j-1)} \ar[dr]_{(j+1,\overline{j+2})} &{\;\;\;\;\;\ldots 32^p\;\;\;\;\;} \ar[d]^{(j-1,j)}\\
{\;\;\;\;\;\ldots 2^{p+1}1^3 \;\;\;\;\;} \ar[dr]_{(j+1,\overline{j+2})}&& {\;\;\;\;\;\ldots 32^{p-1}1^2\;\;\;\;\;} \ar[dl]^{(i,j-1)}\\
&{\ldots 2^{p+1}1 }
} }
\]

Now consider two distinct covers $\mu\lessdot\nu$ and $\mu\lessdot\pi$ of $\mu$, labeled $(i,j)$ and $(k,\overline{k+1})$, respectively.

\begin{proposition}\label{intcup} An upper bound of $\nu$ and $\pi$ is always obtained as in one of the Cases {\rm C1$-$C2}, now denoted {\rm C$1'-$C$2'$}.
\end{proposition}

\begin{proof} We have $\mu_k=\mu_{k+1}=0$ and $\pi_k=\pi_{k+1}=1$. Clearly, we must have $j<k$ and $\mu_{k-1}>0$. If $j<k-1$, or $j=k-1$ and $\mu_{k-1}>1$, we are in Case C1 (a). Otherwise, when $j=k-1$ and $\mu_{k-1}=1$, we have either $i<j-1$ or $i=j-1$; but these are precisely Cases C1 (b) and C2, respectively. 
\end{proof}

Furthermore, we claim that the structure of the interval $[\mu,\nu\vee\pi]$ is always given by the diagrams in Cases C1 or C2. The proof is again completely similar to that of \cite[Proposition~3.2]{brylip}. However, Proposition~\ref{intcup} suffices for our purposes.

\subsection{Type $D_n$}\label{domwd} Now $\lambda$ is a sequence $(\lambda_1\ge\ldots\ge\lambda_n)$ with $\lambda_i\in \frac{1}{2}{\mathbb Z}$, all congruent $\mathrm{mod}\;{\mathbb Z}$, such that $\lambda_{n-1}+\lambda_n\ge 0$. By \cite[Theorem~2.8]{stepod}, a cocover in the corresponding partial order on dominant weights is either of the same form \eqref{cocovera} as in type $A$ (with entries in $\frac{1}{2}{\mathbb Z}$ now allowed), or has one of the following three forms:
\begin{align}\label{cocoverd}
&(\ldots\, 1^20^{n-k-1})\gtrdot(\ldots\, 0^{n-k+1})\,,\;\;\;\;\;(\ldots\, (a+1)a^{n-l-1}(-a+1))\gtrdot(\ldots\,a^{n-l}(-a))\,,
\\&(\ldots\,(a+1)(b+1))\gtrdot(\ldots\,ab)\,,\nonumber
\end{align}
where $1\le k,l\le n-1$ and $a\ge\frac{1}{2}$; these cocovers are labeled by the roots $\varepsilon_k+\varepsilon_{k+1}$, $\varepsilon_l+\varepsilon_n$, and $\varepsilon_{n-1}+\varepsilon_n$, respectively.

We will now assume that $\lambda_i\in{\mathbb Z}$. This implies that the interval $[\widehat{0},\lambda]$ only contains weights $\mu=(\mu_1\ge\ldots\ge\mu_n)$ with $\mu_i\in{\mathbb Z}$. Note that, in this case, there are the same possibilities for the minimal element $\widehat{0}$ as in type $C$, see Section~\ref{domwc}. 

\begin{proposition}\label{stable-crit-d}
If $n>|\lambda|$, then the first cocover in {\rm \eqref{cocoverd}} is the only one which can appear in the Hasse diagram of the interval $[\widehat{0},\lambda]$ beside the type $A$ cocovers in {\rm \eqref{cocovera}}.  
\end{proposition}

\begin{proof}
Assume that the statement fails. Then there is $\mu\le\lambda$ with $\mu_{n-2}\ge 1$, $\mu_1\ge 2$, and $\mu_{n-1}+\mu_n\ge 1$. Combining this with the fact that $|\lambda|\ge|\mu|$, we obtain $|\lambda|\ge n$. But this contradicts the condition in the proposition. 
\end{proof}

From now on, we work under the assumption of Proposition~\ref{stable-crit-d}, and we call this the {\em type $D$ stable range}. Clearly, all the results about the type $C$ stable range in Section~\ref{domwc} apply to the type $D$ one. 


\subsection{Type $B_n$}\label{domwb}
We use the same notation as in Section~\ref{domwc}, and assume that $\lambda$ is a partition $(\lambda_1\ge\ldots\ge\lambda_{n}\ge 0)$, where $\lambda_i\in\mathbb{Z}$. The unique minimal element below $\lambda$ is clearly always $0$.
By \cite[Theorem~2.8]{stepod}, a cocover in the corresponding partial order on dominant weights is either of the same form \eqref{cocovera} as in type $A$, or has one of the following two forms:

\begin{equation}\label{cocoverb}
(\ldots\, 10^{n-k})\gtrdot(\ldots\, 0^{n-k+1})\,,\;\;\;\;\;(\ldots\,(a+1))\gtrdot(\ldots\,a)\,;
\end{equation}
here the first is labeled by $\varepsilon_k$, and the second by $\varepsilon_n$.

The following analogue of Proposition~\ref{stable-crit} is proved in the same way.

\begin{proposition}\label{stable-critb}
If $n>|\lambda|/2$, then the first cocover in {\rm \eqref{cocoverb}} is the only one which can appear in the Hasse diagram of the interval $[\widehat{0},\lambda]$ beside the type $A$ cocovers in {\rm \eqref{cocovera}}.
\end{proposition}

Let us now refer to the intervals $[\nu\wedge\pi,\mu]$ and $[\mu,\nu\vee\pi]$ which can appear in the poset $[\widehat{0},\lambda]$. The intervals involving the new cocover/cover labeled by $(k):=\varepsilon_k$ are completely parallel to those in type $C$, as shown below. We start with distinct cocovers $\mu\gtrdot\nu$ and $\mu\gtrdot\pi$ of a partition $\mu$, which are labeled $(i,j)$ and $(k)$, where necessarily $i<k$.

{\bf Case B1:} cocovers which are (a) nonoverlapping; (b) fully overlapping, with $(i,j)$ proper of type (**). As subcase (a) is easy, only subcase (b) is represented in the diagram below. 

In subcase (b), we have $p\ge 1$, while $i$ is the position of (the shown) $2$ in the partition $\mu$, and $j=i+p+1$ is the position of the first $0$. 
\[\scriptstyle
{\xymatrix{
&{\ldots 21^p}\ar[dl]_{(i,j)}\ar[dr]^{(j-1)}\ar@<0.2mm>[dl]\ar@<0.4mm>[dl]\ar@<0.2mm>[dr]\ar@<0.4mm>[dr]\\
{\ldots 1^{p+2}}\ar[dr]_{(j)}&&{\ldots 21^{p-1}}\ar[dl]^{(i,j-1)}\\
&{\ldots 1^{p+1}}
} }
\]

{\bf Case B2:} partially overlapping cocovers, with $(i,j)$ of type (*), so $j=i+1$. We have $a\ge 3$, while $i$ is the position of $a$ in the partition $\mu$. 
\[\scriptstyle
{\xymatrix{
&{\ldots a1}\ar[dl]_{(i,i+1)}\ar[ddr]^{(i+1)}\ar@<0.2mm>[dl]\ar@<0.4mm>[dl]\ar@<0.2mm>[ddr]\ar@<0.4mm>[ddr]\\
{\ldots(a-1)2}\ar[dd]_{(i+1,i+2)}\\
&&{\ldots a}\ar[ddl]^{(i,i+1)}\\
{\ldots(a-1)1^2}\ar[dr]_{(i+2)}\\
&{\ldots(a-1)1}} }
\]

{\bf Case B3:} partially overlapping cocovers, with $(i,j)$ proper of type (**). Here $p\ge 1$, while $i$ is the position of (the shown) $3$ in the partition $\mu$,  and $j=i+p+1$ is the position of $1$. 
\[\scriptstyle
{\xymatrix{
&{\ldots 32^p1} \ar[d]_{(j-1,j+1)}\ar[dl]_{(i,j)}\ar[dr]^{(j)}\ar@<0.2mm>[dl]\ar@<0.4mm>[dl]\ar@<0.2mm>[dr]\ar@<0.4mm>[dr] \\
{\;\;\;\;\;\ldots 2^{p+2} \;\;\;\;\;} \ar[d]_{(j,j+1)} &{\;\;\;\;\;\ldots 32^{p-1}1^3 \;\;\;\;\;} \ar[dl]^{(i,j-1)} \ar[dr]_{(j+1)} &{\;\;\;\;\;\ldots 32^p\;\;\;\;\;} \ar[d]^{(j-1,j)}\\
{\;\;\;\;\;\ldots 2^{p+1}1^2 \;\;\;\;\;} \ar[dr]_{(j+1)}&& {\;\;\;\;\;\ldots 32^{p-1}1^2\;\;\;\;\;} \ar[dl]^{(i,j-1)}\\
&{\ldots 2^{p+1}1 }
} }
\]

The following result is proved in a completely similar way to the corresponding one in type $C$.

\begin{proposition}\label{intb} The lower bound $\nu\wedge\pi$ for $\mu\gtrdot\nu,\pi$ is always obtained as in one of the Cases~{\rm B1$-$B3}. For similar covers $\mu\lessdot\nu,\pi$, the upper bound $\nu\vee\pi$ is always obtained as in one of the Cases {\rm B1$-$B2}, now denoted {\rm B$1'-$B$2'$}.
\end{proposition}

\section{Modified crystal operators}

In this section we define the modified crystal operators and study several properties of them. The definitions and a few basic facts are stated in Section~\ref{sec:modcr} for an arbitrary simple Lie algebra~$\mathfrak{g}$ of rank $r$. For further properties, we restrict to a classical Lie algebra. Throughout, we use the standard labeling of the corresponding Dynkin diagrams. 

\subsection{Definition of the modified crystal operators and basic facts}\label{sec:modcr}
These operators are indexed by arbitrary roots in the Weyl group orbit $W\alpha_1$ of the simple root $\alpha_1$\footnote{Observe
that the orbit $W\alpha_{1}$ coincides with the set of all roots only for
simply laced root systems.}. 
Given such a root $\alpha$, consider the shortest length element in $W$ satisfying $w(\alpha_1)=\alpha$. We define the modified crystal operators $\mathrm{f}_{\alpha}$ and $\mathrm{e}_{\alpha}$ as the conjugations
\begin{equation}\label{defmodcr}\mathrm{f}_{\alpha}:=w\tilde{f}_{1}w^{-1}\,,\;\;\;\;\;\mathrm{e}_{\alpha}:=w\tilde{e}_{1}w^{-1}\end{equation}
 of the ordinary crystal operators $\tilde{f}_{1}$ and $\tilde{e}_{1}$ by
the Kashiwara action of $w$ on $B(\lambda)$ \cite{kascbm}. 
This means that $\mathrm{f}%
_{\alpha}(b)=0$ precisely when $\tilde{f}_{1}$ applied to $w^{-1}(b)$ is $0$. 
When
$\alpha=\alpha_{i}$ is a simple root, we simply write $\mathrm{f}%
_{i}:=\mathrm{f}_{\alpha_{i}}$ and $\mathrm{e}_{i}:=\mathrm{e}_{\alpha_{i}}$.
We clearly have $\mathrm{f}_{1}=\tilde{f}_{1}$, but $\mathrm{f}_{i}\neq
\tilde{f}_{i}$ in general, and similarly for $\mathrm{e}_{i}$\footnote{The operators $\mathrm{f}_{i}$ 
were considered in \cite[Remark~7.4.2]{kascbm}, and it was observed
there that they do not coincide with the ordinary ones, but to the authors'
knowledge they were not studied further.}.

It is not hard to see that the above choice of $w$ implies that $ws_1$ is the shortest length element mapping $\alpha_1$ to $-\alpha$, so we have
\[\mathrm{f}_{-\alpha}=ws_1\tilde{f}_1s_1w^{-1}=w\tilde{e}_{1}w^{-1}=\mathrm{e}_{\alpha}\,.\]
Furthermore, it is easy to check that $\mathrm{f}_{\alpha}(b)=b'$ if and only if $\mathrm{e}_{\alpha}(b')=b$, and that
\[\mathrm{wt}(\mathrm{f}_{\alpha}(b))=\mathrm{wt}(b)-\alpha\,.\]

More generally, if we choose any $v\in W$ such that $v\alpha_1=\alpha$, we can consider the operator $\mathrm{f}_v:=v\tilde{f}_1 v^{-1}$, and ask if it coincides with $\mathrm{f}_\alpha:=\mathrm{f}_w$. We address this question in Remark~\ref{fv} below.

\begin{remark}\label{fv} {\rm First note that $w^{-1}v$ belongs to the stabilizer $W_{\alpha_1}$ of $\alpha_1$, which (as a parabolic subgroup) is generated by the reflections $s_\alpha$ with $\alpha$ orthogonal to $\alpha_1$. In type $A$, we have $W_{\alpha_1}=\langle s_3,\ldots,s_r\rangle$, so $w^{-1}v$ commutes with $\tilde{f}_1$, and thus $\mathrm{f}_v=\mathrm{f}_\alpha$. More generally, the same holds whenever $w^{-1}v$ belongs to $\langle s_3,\ldots,s_r\rangle$. However, in general $\mathrm{f}_v\ne\mathrm{f}_\alpha$. For instance, in types $B$, $C$, and $D$ we have $\alpha_1=\varepsilon_1-\varepsilon_2$, so $W_{\alpha_1}=\langle s_{\varepsilon_1+\varepsilon_2}\rangle\times\langle s_3,\ldots,s_r\rangle$, where $s_{\varepsilon_1+\varepsilon_2}$ does not commute with $\tilde{f}_1$. }
\end{remark}

We endow the vertices of $B(\lambda)$ with the structure of a colored directed
graph $\mathbb{B}(\lambda)$ with edges $b\overset{\alpha}{\dashrightarrow
}b'$ when $b'=\mathrm{f}_{\alpha}(b)$ for a positive root $\alpha\in W\alpha_1$. As noted above, the
graph $\mathbb{B}(\lambda)$ is different from the Kashiwara crystal
$B(\lambda)$, and in fact, unlike the latter, the former is not connected in general.

\begin{lemma}
\label{fenonzero} \label{Lemma_f_dominant}Consider $b\in\mathbb{B}(\lambda)$
and a positive root $\alpha\in W\alpha_{1}$.

\begin{enumerate}
\item If $\langle\mathrm{wt}(b),\alpha\rangle>0$, then $\mathrm{f}_{\alpha
}(b)\neq0$. In particular, if $\mathrm{wt}(b)-\alpha$ is 
dominant, where $\alpha\in R^{+}$, then $\mathrm{f}_{\alpha}(b)\neq0$.

\item If $\langle\mathrm{wt}(b),\alpha\rangle<0$, then $\mathrm{e}_{\alpha
}(b)\neq0$.
\end{enumerate}
\end{lemma}

\begin{proof}
Consider $w\in W$ of smallest length such that $w(\alpha_{1})=\alpha$. If $\langle\mathrm{wt}%
(b),\alpha\rangle>0$, then 
\[\langle w^{-1}\mathrm{wt}(b),w^{-1}(\alpha
)\rangle=\langle\mathrm{wt}(w^{-1}(b)),\alpha_{1}\rangle>0\,,\]
 which implies
that $\tilde{f}_{1}(w^{-1}(b))\neq0$. But this is equivalent to $\mathrm{f}%
_{\alpha}(b)\neq0$. For the second part of (1), let $\mu:=\mathrm{wt}(b)$, and
observe that $\langle\mu,\alpha^{\vee}\rangle=\langle\mu-\alpha,\alpha^{\vee
}\rangle+2\ge2$; so we can apply the first part. The proof of (2) is completely similar.
\end{proof}


\subsection{Properties of the modified crystal operators in classical types}

From now on we assume that the underlying root system is of classical type.

\begin{theorem}
\label{Th_comm_op}Consider two positive roots $\alpha$ and $\beta$ in $W\alpha_{1}$ and a vertex
$b$ in $\mathbb{B}(\lambda)$ such that $\langle\mathrm{wt}(b),\alpha\rangle>0$
and $\langle\mathrm{wt}(b),\beta\rangle>0$. 
\begin{enumerate}
\item Assume that the pair $(\alpha,\beta)$ satisfies: {\rm (i)} it is $(\varepsilon_i-\varepsilon_j,\,\varepsilon_j\pm\varepsilon_k)$ or $(\varepsilon_j\pm\varepsilon_k,\,\varepsilon_i-\varepsilon_j)$, for $i<j<k$; {\rm (ii)} it is $(\varepsilon_{j-1}+\varepsilon_j,\,\varepsilon_i-\varepsilon_j)$ for $i<j-1$, and $\langle\mathrm{wt}(b)-\beta,\varepsilon_{j-1}-\varepsilon_j\rangle=0$. Then we have $\mathrm{f}_{\alpha}\mathrm{f}_{\beta}(b)=\mathrm{f}_{\alpha+\beta}(b)\neq0\,$.
\item Assume that the pair $(\alpha,\beta)$ is in the $W$-orbit of $(\alpha_1,\alpha_3)$, where $\alpha_3=\varepsilon_3-\varepsilon_4$. Then $\mathrm{f}_{\alpha}%
\mathrm{f}_{\beta}(b)=\mathrm{f}_{\beta}\mathrm{f}_{\alpha}(b)\neq0\,$.
\end{enumerate}
\end{theorem}

We first reduce our theorem to an equivalent simpler statement. Its two parts
will be proved in Sections~\ref{sec:red1} and \ref{sec:red2}, respectively.

\begin{lemma}
\label{PropRed} The two parts of Theorem {\rm \ref{Th_comm_op}} follow from the
two statements below, respectively.

\begin{enumerate}
\item For any $b$ such that $\langle\mathrm{wt}(b),\alpha_{1}\rangle>0$ and
$\langle\mathrm{wt}(b),\alpha_{2}\rangle>0$, we have%
\[
\mathrm{f}_{1}\mathrm{f}_{2}(b)=\mathrm{f}_{2}\mathrm{f}_{1}(b)=\mathrm{f}%
_{\alpha_{1}+\alpha_{2}}(b)\neq0\,,
\]
where $\alpha_2=\varepsilon_2-\varepsilon_3$.

\item For any $b$ such that $\langle\mathrm{wt}(b),\alpha_{1}\rangle>0$ and
$\langle\mathrm{wt}(b),\alpha_{3}\rangle>0$, we have%
\[
\mathrm{f}_{1}\mathrm{f}_{3}(b)=\mathrm{f}_{3}\mathrm{f}_{1}(b)\neq0\,,
\]
where $\alpha_3=\varepsilon_3-\varepsilon_4$.
\end{enumerate}
\end{lemma}

\begin{proof}
Given a signed permutation $w$, we use the notation $w[a,b,\ldots]:=[w(a),w(b),\ldots]$. 

We first address Theorem~\ref{Th_comm_op}~(1), and start with case (i), where $(\alpha,\beta)=(\varepsilon_i-\varepsilon_j,\,\varepsilon_j\pm\varepsilon_k)$. Consider the shortest length element $w\in W$ with $w[1,2,3]=[i,j,\pm k]$, as well as $u:=s_1 s_2$ and $v:=s_2$. Assertion~(1) of the lemma implies
\begin{align}\label{conjw}
&\mathrm{f}_1\mathrm{f}_u(b)=\mathrm{f}_u\mathrm{f}_1(b)=\mathrm{f}_v(b)\ne0\;\Longrightarrow \nonumber\\
&(w\mathrm{f}_1w^{-1})(w\mathrm{f}_uw^{-1})(b')=(w\mathrm{f}_uw^{-1})(w\mathrm{f}_1w^{-1})(b')=w\mathrm{f}_vw^{-1}(b')\ne0\;\Longrightarrow \nonumber\\
&\mathrm{f}_w\mathrm{f}_{wu}(b')=\mathrm{f}_{wu}\mathrm{f}_w(b')=\mathrm{f}_{wv}(b')\ne0\,,
\end{align}
where $b':=w(b)$. The condition on weights in Assertion~1 implies $\langle\mathrm{wt}(b'),\alpha\rangle>0$ and $\langle\mathrm{wt}(b'),\beta\rangle>0$. On another hand, note that $wu[1,2]=[j,\pm k]$ and $wv[1,2]=[i,\pm k]$. Moreover, one can check by using Remark~\ref{fv} that we are in the situation where  $\mathrm{f}_w=\mathrm{f}_\alpha$, $\mathrm{f}_{wu}=\mathrm{f}_\beta$, and $\mathrm{f}_{wv}=\mathrm{f}_{\alpha+\beta}$. By plugging into~\eqref{conjw}, the proof in case (i) is concluded. 

The statement in Theorem~\ref{Th_comm_op}~(2) is proved in a completely similar way, based on Assertion~(2) of the lemma.

We now turn to case (ii) in Theorem~\ref{Th_comm_op}~(1). By a similar reasoning as above, the statement follows from the special case  $(\alpha,\beta)=(\varepsilon_2+\varepsilon_3,\,\varepsilon_1-\varepsilon_3)$. Consider the shortest length element $w\in W$ with $w[1,2,3]=[1,3,\overline{2}]$, as well as $u:=s_1 s_2$ and $v:=s_2$, like before. By using conjugation as above, Assertion~(1) of the lemma implies 
\begin{equation}\label{conjw1}\mathrm{f}_{wu}\mathrm{f}_{w}(b')=\mathrm{f}_{wv}(b')\ne0\,.\end{equation}
Note that $wu[1,2]=[3,\overline{2}]$ and $wv[1,2]=[1,\overline{2}]$. Therefore, by Remark~\ref{fv}, we have $\mathrm{f}_w=\mathrm{f}_\beta$ and $\mathrm{f}_{wv}=\mathrm{f}_{\alpha+\beta}$, but $\mathrm{f}_{wu}\ne\mathrm{f}_\alpha$. In fact, the shortest coset representative in $wuW_{\alpha_1}$ is $wus_{\varepsilon_1+\varepsilon_2}=s_2wu$, so $\mathrm{f}_\alpha=s_2\mathrm{f}_{wu}s_2$. However, letting $b'':=\mathrm{f}_\beta(b')$, where $\rm{wt}(b'')=\rm{wt}(b')-\beta$, we can see that the condition $\langle\mathrm{wt}(b''),\varepsilon_{2}-\varepsilon_3\rangle=0$ implies
\[\mathrm{f}_\alpha(b'')=s_2\mathrm{f}_{wu}s_2(b'')=\mathrm{f}_{wu}(b'')\,.\]
The proof is concluded by plugging into~\eqref{conjw1}.
\end{proof}

We have an analogous result to Theorem~\ref{Th_comm_op} for the $\mathrm{e}%
_{\cdot}$ operators.

\begin{theorem}
\label{Th_comm_ope}Consider two positive roots $\alpha$ and $\beta$ in $W\alpha_{1}$ and a
vertex $b$ in $\mathbb{B}(\lambda)$ such that $\langle\mathrm{wt}%
(b),\alpha\rangle\geq0$ and $\langle\mathrm{wt}(b),\beta\rangle\geq0$. Assume
also that $\mathrm{e}_{\alpha}(b)\neq0$ and $\mathrm{e}_{\beta}(b)\neq0$.
\begin{enumerate}
\item Assume that the pair $(\alpha,\beta)$ satisfies: {\rm (i)} it is $(\varepsilon_i-\varepsilon_j,\,\varepsilon_j\pm\varepsilon_k)$ or $(\varepsilon_j\pm\varepsilon_k,\,\varepsilon_i-\varepsilon_j)$, for $i<j<k$; {\rm (ii)} it is $(\varepsilon_i-\varepsilon_j,\,\varepsilon_{j-1}+\varepsilon_j)$ for $i<j-1$, and $\langle\mathrm{wt}(b),\varepsilon_{j-1}-\varepsilon_j\rangle=0$. Then we have $\mathrm{e}_{\alpha}\mathrm{e}_{\beta}(b)=\mathrm{e}_{\alpha+\beta}(b)\neq0\,$.
\item Assume that the pair $(\alpha,\beta)$ is in the $W$-orbit of $(\alpha_1,\alpha_3)$, where $\alpha_3=\varepsilon_3-\varepsilon_4$, and $w$ is a shortest length element satisfying $w(\alpha_1,\alpha_3)=(\alpha,\beta)$. Let $\gamma:=w(\alpha_2)$, and also assume that $\langle\mathrm{wt}(b),\gamma\rangle>0$. Then $\mathrm{e}_{\alpha
}\mathrm{e}_{\beta}(b)=\mathrm{e}_{\beta}\mathrm{e}_{\alpha}(b)\neq0$.
\end{enumerate}
\end{theorem}

By analogy with Theorem~\ref{Th_comm_op}, the above result is proved based on
the following reduction. In turn, the two parts of Lemma~\ref{PropRede} are
proved in the same way as those of Lemma~\ref{PropRed}, also in
Sections~\ref{sec:red1} and \ref{sec:red2}, respectively.

\begin{lemma}
\label{PropRede} The two parts of Theorem {\rm \ref{Th_comm_ope}} follow from the
two statements below, respectively.

\begin{enumerate}
\item For any $b$ such that $\langle\mathrm{wt}(b),\alpha_{1}\rangle\ge0$,
$\langle\mathrm{wt}(b),\alpha_{2}\rangle\ge0$, $\mathrm{e}_{1}(b)\ne0$, and
$\mathrm{e}_{2}(b)\ne0$, we have%
\[
\mathrm{e}_{1}\mathrm{e}_{2}(b)=\mathrm{e}_{2}\mathrm{e}_{1}(b)=\mathrm{e}%
_{\alpha_{1}+\alpha_{2}}(b)\neq0\,,
\]
where $\alpha_2=\varepsilon_2-\varepsilon_3$.

\item For any $b$ such that $\langle\mathrm{wt}(b),\alpha_{1}\rangle\geq0$,
$\langle\mathrm{wt}(b),\alpha_{2}\rangle>0$, $\langle\mathrm{wt}(b),\alpha
_{3}\rangle\geq0$, $\mathrm{e}_{1}(b)\neq0$ and $\mathrm{e}_{3}(b)\neq0$, we
have%
\[
\mathrm{e}_{1}\mathrm{e}_{3}(b)=\mathrm{e}_{3}\mathrm{e}_{1}(b)\neq0\,,
\]
where $\alpha_3=\varepsilon_3-\varepsilon_4$.
\end{enumerate}
\end{lemma}

\begin{proof} The proof is completely similar to that of Lemma~\ref{PropRed}.
\end{proof}

As mentioned, we start with a direct proof of the first parts of
Lemmas~\ref{PropRed} and \ref{PropRede}; these are based on the description of
the actions of $\mathrm{f}_{2}$ and $\mathrm{f}_{\alpha_{1}+\alpha_{2}}$ on
semistandard tableaux. We then prove the second parts by using the
\emph{cyclage} of Lascoux and Sch\"{u}tzenberger (which can also be used to
reprove the first parts) on type $A_{3}$-tableaux.

\subsection{Cyclage, charge and the modified crystal operators}\label{sec:cyccharge}

Given a semistandard tableau $T$, write $C(T)=x\hookrightarrow T^{\flat}$
where $x\hookrightarrow T^{\flat}$ is the semistandard tableau obtained after
the row-insertion of the south-west letter $x$ of $T$ in the tableau
$T\setminus\{x\}$. The combinatorial procedure $T\rightarrow C(T)$ is called
the cyclage of the tableau $T$. It is known (see \cite{LLT}) that the sequence
of cyclages applied to $T$ will eventually lead to the unique row-tableau
$R_{\mu}$ where $\mu=\mathrm{wt}(T)$ is the weight of $T$. The number
$\mathrm{co}(T)$ of cyclage operations used in this sequence is called the
cocharge of $T$.\ The charge of $T$ is then defined as $\mathrm{c}%
(T)=\left\Vert \mu\right\Vert -\mathrm{co}(T)$ where $\left\Vert
\mu\right\Vert =\sum_{i=1}^{n-1}(i-1)\mu_{i}$.

\begin{example}{\rm 
For $T=%
\begin{tabular}
[c]{|l|ll}\hline
$1$ & $1$ & \multicolumn{1}{|l|}{$4$}\\\hline
$2$ & $2$ & \multicolumn{1}{|l}{}\\\cline{1-1}\cline{1-2}%
$3$ &  & \\\cline{1-1}%
\end{tabular}
\ \ $ we get}
\begin{align*}
T_{1} &  =C(T)=%
\begin{tabular}
[c]{|l|l|l|}\hline
$1$ & $1$ & $3$\\\hline
$2$ & $2$ & $4$\\\hline
\end{tabular}
\ \ ,T_{2}=C^{2}(T)=%
\begin{tabular}
[c]{|l|ll}\hline
$1$ & $1$ & \multicolumn{1}{|l|}{$2$}\\\hline
$2$ & $3$ & \multicolumn{1}{|l}{}\\\cline{1-1}\cline{1-2}%
$4$ &  & \\\cline{1-1}%
\end{tabular}
\ \ ,T_{3}=C^{3}(T)=%
\begin{tabular}
[c]{|l|l|ll}\hline
$1$ & $1$ & $2$ & \multicolumn{1}{|l|}{$4$}\\\hline
$2$ & $3$ &  & \\\cline{1-2}%
\end{tabular}
\\
T_{4} &  =C^{4}(T)=%
\begin{tabular}
[c]{|l|l|ll}\hline
$1$ & $1$ & $2$ & \multicolumn{1}{|l|}{$2$}\\\hline
$3$ & $4$ &  & \\\cline{1-2}%
\end{tabular}
\ \ ,T_{5}=C^{5}(T)=%
\begin{tabular}
[c]{|l|llll}\hline
$1$ & $1$ & \multicolumn{1}{|l}{$2$} & \multicolumn{1}{|l}{$2$} &
\multicolumn{1}{|l|}{$3$}\\\hline
$4$ &  &  &  & \\\cline{1-1}%
\end{tabular}
\\
T_{6} &  =C^{6}(T)=%
\begin{tabular}
[c]{|l|l|l|l|l|l|}\hline
$1$ & $1$ & $2$ & $2$ & $3$ & $4$\\\hline
\end{tabular}
\ \ .
\end{align*}
{\rm Therefore $\mathrm{co}(T)=6$ and $\mathrm{c}(T)=7-6=1$.}
\end{example}

We can endow the set $\mathrm{Tab}_{\mu}$ of semistandard tableaux of weight
$\mu$ with the structure of an oriented graph with an oriented edge
$T\rightsquigarrow T'$ when $T'=C(T)$.\ We then have a unique
sink vertex in $\mathrm{Tab}_{\mu}$ which is the row tableau of evaluation
$\mu$.\ Given $\mu$ and $\nu$ two weights, we write $\nu\leq\mu$ when $\mu
-\nu$ can be written as a linear combination of simple roots $\varepsilon
_{i}-\varepsilon_{i+1},i=1,\ldots,n-1$ with nonnegative integral
coefficients.\ We refer to \cite{LLT} for a proof of the following theorem.

\begin{theorem}
\label{Th_fund} \ 

\begin{enumerate}
\item For any $\sigma\in\mathfrak{S}_{n}$ the Kashiwara action $T\longmapsto
\sigma(T)$ gives an isomorphism of oriented graphs from $\mathrm{Tab}_{\mu}$
to $\mathrm{Tab}_{\sigma(\mu)}$.

\item Assume $\mu_{1}>\mu_{2}$. Then the action of the Kashiwara crystal
operator $\tilde{f}_{1}$ yields an embedding of oriented graphs from
$\mathrm{Tab}_{\mu}$ to $\mathrm{Tab}_{\mu-\alpha_{1}}$.

\item If $C(T)=C(T')$ where $T$ and $T'$ are two tableaux with
the same shape, then $T=T'$. Thus the embedding in {\rm (2)} is the
unique which preserves the shape of the tableaux.
\end{enumerate}
\end{theorem}

\begin{corollary}
\label{Cor_fc=cf}Consider a tableau $T$ and a positive root $\alpha$ such that
$\mathrm{f}_{\alpha}(T)\neq0$. Then $\mathrm{f}_{\alpha}(C(T))=C(\mathrm{f}%
_{\alpha}(T))\neq0$.\footnote{Nevertheless, there are tableaux $T$ such that
$\mathrm{f}_{\alpha}(C(T))\neq0$ but $\mathrm{f}_{\alpha}(T)=0$.}
\end{corollary}

It is well-know that the charge statistic yields a combinatorial description
of the Kostka polynomials in type $A$.

\begin{theorem}
\label{TH_LScharge}For any partitions $\lambda$ and $\mu$ we have
\[
K_{\lambda,\mu}(t)=\sum_{T\in\mathrm{Tab}(\lambda)_{\mu}}t^{\mathrm{c}(T)}%
\]
where $\mathrm{Tab}(\lambda)_{\mu}$ is the set of semistandard tableaux of
shape $\lambda$ and evaluation $\mu$.
\end{theorem}

The following proposition is a consequence of Theorems~\ref{Th_fund} and
\ref{TH_LScharge}. It shows how the charge and cocharge statistics are
modified when a modified crystal operator is applied to a tableau $T$ of shape
$\lambda$ (regarded as a vertex of the crystal $B(\lambda)$). For any positive
root $\alpha$, let $\left\vert \alpha\right\vert $ be the height of $\alpha$,
that is the number of simple roots appearing in the decomposition of $\alpha$
as a sum of simple roots.

\begin{proposition}
\label{Prop_Modop_charge}Let $T$ be a vertex of $B(\lambda)$ of weight $\mu$
and $\alpha$ a positive root such that $\mathrm{f}_{\alpha}(T)\neq0$. Then
\[
\mathrm{co}(\mathrm{f}_{\alpha}(T))=\mathrm{co}(T)\text{ and }\mathrm{c}%
(\mathrm{f}_{\alpha}(T))=\mathrm{c}(T)+\left\vert \alpha\right\vert .
\]
\end{proposition}

\begin{proof}
Since $\mathrm{f}_{\alpha}$ is obtained by conjugation of the action of
$\tilde{f}_{1}$, we get the equality $\mathrm{co}(\mathrm{f}_{\alpha
}(T))=\mathrm{co}(T)$ from Assertion 1 and 2 of Theorem~\ref{Th_fund}.\ Write
$\alpha=\varepsilon_{i}-\varepsilon_{i}$ with $1<i<j<n$. Then $\mathrm{f}%
_{\alpha}(T)$ has weight $\mu-\alpha$. Since $\mathrm{co}(\mathrm{f}_{\alpha
}(T))=\mathrm{co}(T)$, we have
\begin{align*}
\mathrm{c}(\mathrm{f}_{\alpha}(T))-\mathrm{c}(T)&=\left\Vert \mu-\alpha
\right\Vert -\left\Vert \mu\right\Vert \\
&=(i-1)((\mu_{i}-1)-\mu_{i})+(j-1)((\mu_{j}+1)-\mu_{j})=j-i=\left\vert
\alpha\right\vert\,.
\end{align*}
\end{proof}

\subsection{Proof of Lemmas \ref{PropRed}~(1) and \ref{PropRede}~(1)}
\label{sec:red1}

Observe first that it suffices to prove Lemmas~\ref{PropRed}~(1) and
\ref{PropRede}~(1) in type $A_{2}$. To do this, we will use the tableau
realization of the crystal $B(\lambda)$. As we restrict to type $A_{2}$, we
view $\lambda$ as a partition $\lambda=(\lambda_{1}\geq\lambda_{2}\geq0)$, and
we index the vertices of $B(\lambda)$ by semistandard tableaux of shape
$\lambda$ with the alphabet $\{1<2<3\}$. We adopt the English notation for
semistandard tableaux. The row reading of a semistandard tableau $T$ is the
word $\mathrm{w}(T)$ obtained by reading its rows from right to left and top
to bottom. The actions of $\mathrm{f}_{1}$, $s_{\alpha_{1}}$, and
$s_{\alpha_{2}}$ on $B(\lambda)$ are computed by using the following classical procedures.

For any $i=1,2$, write $w_{i}$ for the subword of $\mathrm{w}(T)$ formed by
the letters in $\{i,i+1\}$. Let $w_{i}^{\mathrm{red}}=(i+1)^{r}i^{s}$ be the
subword of $w_{i}$ obtained by recursively deleting factors $i(i+1)$.\ Now
consider $w_{i}^{\mathrm{red}}$ as a subword of $\mathrm{w}(T)$.\ If $r>s$,
then $s_{\alpha_{i}}$ is obtained by replacing in $\mathrm{w}(T)$ the $r-s$
rightmost letters $i+1$ of $w_{i}^{\mathrm{red}}$ with $i$. Otherwise,
$s_{\alpha_{i}}$ is obtained by replacing in $\mathrm{w}(T)$ the $s-r$
leftmost letters $i$ of $w_{i}^{\mathrm{red}}$ with $i+1$. Observe here that
for each factorization $w_{i}=(i+1)^{a}u_{i}i^{a}$, we have $s_{i}%
(w_{i})=(i+1)^{a}s_{i}(u_{i})i^{a}$.

The action of $\mathrm{f}_{1}$ on a tableau $T$ is straightforward. If $T$
does not contain any letter $1$, then $\mathrm{f}_{1}(T)=0$. Otherwise
$\mathrm{f}_{1}(T)$ is obtained by replacing its rightmost letter $1$ (always
in its first row) with $2$.

Let us now describe the actions of $\mathrm{f}_{\alpha_{1}+\alpha_{2}}$ and
$\mathrm{f}_{2}$ on a tableau $T$. Let $m_{i}^{(k)}=m_{i}^{(k)}(T)$ be the
number of letters $i$ in the $k$-th row of ${T}$. Let $\mu=(\mu_{1},\mu
_{2},\mu_{3})$ be the \emph{content} of $T$, i.e., $\mu_{i}$ is the number of
letters $i$ in $T$.

\begin{lemma}
\label{f_1,3)}Assume $\mu_{1}>\mu_{3}.$

\begin{enumerate}
\item If $T$ does not contains any letter $2$, then $\mathrm{f}_{\alpha
_{1}+\alpha_{2}}(T)$ is obtained by changing in $T$ the rightmost letter $1$
into a letter $3$.

\item Otherwise $\mathrm{f}_{\alpha_{1}+\alpha_{2}}(T)$ is obtained by
changing in $T$ the rightmost letter $1$ into a letter $2$ and

\begin{itemize}
\item the rightmost letter $2$ in its first row into a letter $3$ when
$m_{2}^{(2)}\leq m_{3}^{(1)},$

\item the rightmost letter $2$ in its second row into a letter $3$ when
$m_{2}^{(2)}>m_{3}^{(1)}.$
\end{itemize}
\end{enumerate}
\end{lemma}

\begin{proof}
Recall that $\mathrm{f}_{\alpha_{1}+\alpha_{2}}(T)=s_{\alpha_{2}}%
\mathrm{f}_{1}s_{\alpha_{2}}(T)$.

Let ${\mathrm{w}}(T)=3^{a}2^{b}1^{c}3^{d}2^{e}$. The assumption $\mu_{1}%
>\mu_{3}$ is written $c>a+d\ge0$. Let $m=\min(a,e)$. By the definition of the
actions of $s_{\alpha_{2}}$ and $\mathrm{f}_{1}$, the contributions of the $m$
leftmost letters $2$ in the second row and that of the $m$ rightmost letters
$3$ in the first row cancel. So, without loss of generality, we have the
following two cases, in which we show the successive actions of $s_{\alpha
_{2}}$, $\mathrm{f}_{1}$, and $s_{\alpha_{2}}$; the paired entries when
applying $s_{\alpha_{2}}$ are underlined.

{\bf Case 1:} $a\ge e=0$.
\begin{align*}
3^{a}\,2^{b}\,1^{c}\,3^{d} & =3^{a}\,2^{\max(b-d,0)}\,\underline{2^{\min
(b,d)}}\,1^{c}\,\underline{3^{\min(b,d)}}\,3^{\max(d-b,0)}%
\xrightarrow{s_{\alpha_{2}}}\\
& \xrightarrow{s_{\alpha_{2}}}3^{\max(b-d,0)}\,2^{a}\,2^{\min(b,d)}%
\,1^{c}\,3^{\min(b,d)}\,2^{\max(d-b,0)}\xrightarrow{\mathrm{f}_{1}}\\
& \xrightarrow{\mathrm{f}_{1}}3^{\max(b-d,0)}\,2^{a+1}\,\underline
{2^{\min(b,d)}}\,1^{c-1}\,\underline{3^{\min(b,d)}}\,2^{\max(d-b,0)}%
\xrightarrow{s_{\alpha_{2}}}\\
& \xrightarrow{s_{\alpha_{2}}} 3^{a+1}\,2^{\max(b-d,0)}\,2^{\min
(b,d)}\,1^{c-1}\,3^{\min(b,d)}\,3^{\max(d-b,0)}=3^{a+1}\,2^{b}\,1^{c-1}%
\,3^{d}\,.
\end{align*}

{\bf Case 2:} $e> a=0$.
\begin{align*}
2^{b}\,1^{c}\,3^{d}\,2^{e} & =2^{\max(b-d,0)}\,\underline{2^{\min(b,d)}%
}\,1^{c}\,\underline{3^{\min(b,d)}}\,3^{\max(d-b,0)}\,2^{e}%
\xrightarrow{s_{\alpha_{2}}}\\
& \xrightarrow{s_{\alpha_{2}}} 3^{\max(b-d,0)}\,2^{\min(b,d)}\,1^{c}%
\,3^{\min(b,d)}\,3^{e}\,2^{\max(d-b,0)}\xrightarrow{\mathrm{f}_{1}}\\
& \xrightarrow{\mathrm{f}_{1}} 3^{\max(b-d,0)}\,\underline{2^{\min(b,d)+1}%
}\,1^{c-1}\,\underline{3^{\min(b,d)+1}}\,3^{e-1}\,2^{\max(d-b,0)}%
\xrightarrow{s_{\alpha_{2}}}\\
& \xrightarrow{s_{\alpha_{2}}} 2^{\max(b-d,0)}\,2^{\min(b,d)+1}\,1^{c-1}%
\,3^{\min(b,d)+1}\,3^{\max(d-b,0)}2^{e-1} =2^{b+1}\,1^{c-1}\,3^{d+1}%
\,2^{e-1}\,.
\end{align*}
\end{proof}

\begin{lemma}
\label{lemma_f_(2)}Assume $\mu_{2}>\mu_{3}$.\ If $m_{2}^{(2)}\leq m_{3}^{(1)}$
(respectively $m_{2}^{(2)}>m_{3}^{(1)}$), the tableau $\mathrm{f}_{2}(T)$ is
obtained by changing in $T$ the rightmost letter $2$ in the first
(respectively second) row into a letter $3$.
\end{lemma}

\begin{proof}
Recall that $\mathrm{f}_{2}(T)=s_{1}\mathrm{f}_{\alpha_{1}+\alpha_{2}}%
s_{1}(T)$.

Let ${\mathrm{w}}(T)=3^{a}2^{b}1^{c}3^{d}2^{e}$, so we have $c\ge e$. The
assumption $\mu_{2}>\mu_{3}$ is written $b+e>a+d$. Based on Lemma~\ref{f_1,3)}%
, we consider the following two cases, in which we show the successive actions
of $s_{\alpha_{1}}$, $\mathrm{f}_{\alpha_{1}+\alpha_{2}}$, and $s_{\alpha_{1}%
}$; the paired entries when applying $s_{\alpha_{1}}$ are underlined.

{\bf Case 1:} $a\ge e$. We have $b+e>a+d\ge d+e$, so $b>d\ge0$.
\begin{align*}
3^{a}\,2^{b}\,1^{c}\,3^{d}\,2^{e} & =3^{a}\,2^{b}\,1^{c-e}\,\underline{1^{e}%
}\,3^{d}\,\underline{2^{e}}\xrightarrow{s_{\alpha_{1}}} 3^{a}\,2^{c-e}%
\,1^{b+e}\,3^{d}\,2^{e}\xrightarrow{\mathrm{f}_{\alpha_{1}+\alpha_{2}}}3^{a+1}%
\,2^{c-e}\,1^{b+e-1}\,3^{d}\,2^{e}=\\
& =3^{a+1}\,2^{c-e}\,1^{b-1}\,\underline{1^{e}}\,3^{d}\,\underline{2^{e}%
}\xrightarrow{s_{\alpha_{1}}}3^{a+1}\,2^{b-1}\,1^{c}\,3^{d}\,2^{e}\,.
\end{align*}

{\bf Case 2:} $e> a\ge0$.
\begin{align*}
3^{a}\,2^{b}\,1^{c}\,3^{d}\,2^{e} & \xrightarrow{s_{\alpha_{1}}}
3^{a}\,2^{c-e}\,1^{b+e}\,3^{d}\,2^{e}%
\xrightarrow{\mathrm{f}_{\alpha_{1}+\alpha_{2}}}3^{a}\,2^{c-e+1}%
\,1^{b+e-1}\,3^{d+1}\,2^{e-1}=\\
& =3^{a}\,2^{c-e+1}\,1^{b}\,\underline{1^{e-1}}\,3^{d+1}\,\underline{2^{e-1}%
}\xrightarrow{s_{\alpha_{1}}}3^{a}\,2^{b}\,1^{c}\,3^{d+1}\,2^{e-1}\,.
\end{align*}
\end{proof}

\begin{corollary}
\label{cor_e_2} Consider a tableau $T$ of weight $(\mu_{1},\mu_{2},\mu_{3})$
such that $\mu_{2}\geq\mu_{3}$. Then $\mathrm{e}_{2}(T)=0$ if and only if
$m_{2}^{(2)}\geq m_{3}^{(1)}$ and $m_{3}^{(2)}=0$.
\end{corollary}

\begin{proof}
Assume $m_{2}^{(2)}<m_{3}^{(1)}$. Since $m_{3}^{(1)}>0$, we can consider the
tableau $T'$ obtained by changing the leftmost letter $3$ in $2$ in
the first row of $T$. By the previous lemma, we have then $\mathrm{e}%
_{2}(T)=T'\neq0$.\ Similarly, when $m_{2}^{(2)}\geq m_{3}^{(1)}$ and
$m_{3}^{(2)}>0$, we also get $\mathrm{e}_{2}(T)\neq T'$ where
$T'$ is then obtained from $T$ by changing the leftmost letter $3$ in
$2$ in the second row. Conversely, when $m_{2}^{(2)}\geq m_{3}^{(1)}$ and
$m_{3}^{(2)}=0$, the tableau $T'$ obtained from $T$ by changing the
leftmost letter $3$ in $2$ in the first row (if any) does not satisfy
$\mathrm{f}_{2}(T')=T$. Thus, we have $\mathrm{e}_{2}(T)=0$.
\end{proof}

The following lemma is a rephrasing of the statement in Lemma~\ref{PropRed}%
~(1). It is an easy consequence of Lemmas~\ref{f_1,3)} and \ref{lemma_f_(2)}.

\begin{lemma}
\label{actf12} If the tableau $T$ is such that $\mu_{1}>\mu_{2}>\mu_{3}$, then
we have
\[
\mathrm{f}_{1}\mathrm{f}_{2}(T)=\mathrm{f}_{2}\mathrm{f}_{1}(T)=\mathrm{f}%
_{\alpha_{1}+\alpha_{2}}(T)\ne0\,.
\]
\end{lemma}

Lemma~\ref{PropRede}~(1) is rephrased as follows.

\begin{lemma}
\label{acte12} If the tableau $T$ is such that $\mu_{1}\ge\mu_{2}\ge\mu_{3}$,
$\mathrm{e}_{1}(T)\ne0$, and $\mathrm{e}_{2}(T)\ne0$, then we have
\[
\mathrm{e}_{1}\mathrm{e}_{2}(T)=\mathrm{e}_{2}\mathrm{e}_{1}(T)=\mathrm{e}%
_{\alpha_{1}+\alpha_{2}}(T)\ne0\,.
\]
\end{lemma}

\begin{proof}
Let $T':=\mathrm{e}_{2}(T)$. We have $\langle\mathrm{wt}(T^{\prime
}),\alpha_{2}\rangle>0$. Therefore, the action of $\mathrm{f}_{2}$ on
$T'$, which produces $T$, is described by Lemma~\ref{lemma_f_(2)}.
Thus, we have the following two cases, where we use the notation $m_{i}%
^{(k)}(\cdot)$ introduced above.

{\bf Case 1:} $m_{2}^{(2)}(T')\leq m_{3}^{(1)}(T')$. In this
case, we have $m_{2}^{(1)}(T)=m_{2}^{(1)}(T')-1$. Since $\mathrm{e}%
_{1}(T)\ne0$, we have $m_{2}^{(1)}(T)>0$, which implies $m_{2}^{(1)}%
(T')>0$. It means that $T^{\prime\prime}:=\mathrm{e}_{1}(T^{\prime
})\ne0$.

{\bf Case 2:} $m_{2}^{(2)}(T')>m_{3}^{(1)}(T')$. In this
case, we have $m_{2}^{(1)}(T)=m_{2}^{(1)}(T')$. In the same way as in
Case 1, we deduce $T^{\prime\prime}:=\mathrm{e}_{1}(T')\ne0$.

In both cases, we have $\mathrm{wt}(T^{\prime\prime})=\mathrm{wt}%
(T)+\alpha_{1}+\alpha_{2}$. Therefore, the hypothesis of Lemma~\ref{actf12} is
satisfied for $T^{\prime\prime}$, and the proof is completed by applying this lemma.
\end{proof}

\subsection{Proof of Lemmas \ref{PropRed}~(2) and \ref{PropRede}~(2)}

\label{sec:red2}

The difficulty is that the second parts of Lemmas~\ref{PropRed} and
\ref{PropRede} reduce to type $A_{3}$, rather than type $A_{2}$, like the
first parts. The action of the modified crystal operator $\mathrm{f}_{3}$ can
also be described in the same spirit as in Lemma~\ref{lemma_f_(2)}, but this
requires the enumeration of numerous configurations. Fortunately, in order to
prove Lemmas~\ref{PropRed}~(2) and \ref{PropRede}~(2), this can be avoided by
using the cyclage on semistandard tableaux.

Lemma~{\ref{PropRed}}~(2) is now rephrased as follows.

\begin{lemma}
\label{actf13} If the tableau $T\ $is such that $\mu_{1}>\mu_{2}$ and $\mu
_{3}>\mu_{4}$, then we have
\[
\mathrm{f}_{1}\mathrm{f}_{3}(T)=\mathrm{f}_{3}\mathrm{f}_{1}(T)\neq0\,.
\]
\end{lemma}

\begin{proof}
By Lemma~\ref{Lemma_f_dominant}, we have $\mathrm{f}_{1}\mathrm{f}_{3}%
(T)\neq0$ and $\mathrm{f}_{3}\mathrm{f}_{1}(T)\neq0$.\ We argue by induction
on the cocharge.\ When $\mathrm{co}(T)=0$, the tableau $T$ has row shape, and
the equality is clear since $\mathrm{f}_{1}\mathrm{f}_{3}(T)=\mathrm{f}%
_{3}\mathrm{f}_{1}(T)$ is the unique row of weight $\mathrm{wt}(T)-\alpha
_{1}-\alpha_{3}$.\ Assume that the lemma holds for any tableau with cocharge
$k-1$, and consider $T$ such that $\mathrm{co}(T)=k$.\ Set $\mathrm{f}%
_{1}\mathrm{f}_{3}(T)=U$ and $\mathrm{f}_{3}\mathrm{f}_{1}(T)=U'$.\ By
Theorem~\ref{Th_fund}~(2), we get $\mathrm{f}_{1}(C(T))=C(\mathrm{f}_{1}(T))$.
Then $\mathrm{f}_{3}\mathrm{f}_{1}(C(T))=\mathrm{f}_{3}C(\mathrm{f}_{1}%
(T))$.\ But the cyclage operation also commutes with $\mathrm{f}%
_{3}\mathrm{\ }$(in fact with any modified crystal operator), because it
commutes with the actions of $W$ and $\mathrm{f}_{1}$, by
Theorem~\ref{Th_fund}~(1),~(2). So we have $\mathrm{f}_{3}\mathrm{f}%
_{1}(C(T))=C(\mathrm{f}_{3}\mathrm{f}_{1}(T))=C(U)$.\ We obtain similarly the
equality $\mathrm{f}_{1}\mathrm{f}_{3}(C(T))=C(\mathrm{f}_{1}\mathrm{f}%
_{3}(T))=C(U')$.\ By our induction hypothesis, we thus deduce that
$C(U)=C(U')$. But since $C(U)$ and $C(U')$ have the same shape
(i.e., that of $C(T)$), Theorem~\ref{Th_fund}~(3) implies the desired equality
$U=U'$.
\end{proof}

Lemma~{\ref{PropRede}}~(2) is rephrased as follows.

\begin{lemma}
\label{acte13} If the tableau $T\ $is such that $\mu_{1}\geq\mu_{2}>\mu
_{3}\geq\mu_{4}$, $\mathrm{e}_{1}(T)\neq0$ and $\mathrm{e}_{3}(T)\neq0$, then
we have
\[
\mathrm{e}_{1}\mathrm{e}_{3}(T)=\mathrm{e}_{3}\mathrm{e}_{1}(T)\neq0\,.
\]
\end{lemma}

We first prove a weaker version for two-row tableaux.

\begin{lemma}\label{lem:tworow}
Lemma~{\rm \ref{acte13}} is true for any two-row tableau $T$.
\end{lemma}

\begin{proof}
Observe first we must have $\mu_{3}>0$. Otherwise $\mu_{3}=\mu_{4}=0$ and we
cannot have $\mathrm{e}_{3}(T)\neq0$ for this would give a tableau of
evaluation $(\mu_{1},\mu_{2},1,-1)$. Set $T'=\mathrm{e}_{1}(T)$. When
$\mathrm{e}_{3}(T')\neq0$, the tableau $T^{\prime\prime}%
=\mathrm{e}_{3}(T')$ has weight $(\mu_{1}+1,\mu_{2}-1,$ $\mu_{3}%
+1,\mu_{4}-1)$ and by applying Lemma~\ref{actf13} to $T^{\prime\prime}$ we
get
\[
T=\mathrm{f}_{1}\mathrm{f}_{3}(T^{\prime\prime})=\mathrm{f}_{3}\mathrm{f}%
_{1}(T^{\prime\prime})\neq0\,,
\]
which is equivalent to the desired equality.

We claim we cannot have $\mathrm{e}_{3}(T')=0$.\ Indeed, assume
$\mathrm{e}_{3}(T')=0$. Recall we have $\mathrm{e}_{3}=s_{2}%
s_{3}\mathrm{e}_{2}s_{3}s_{2}$. We obtain $\mathrm{e}_{2}s_{3}s_{2}(T^{\prime
})=0$ and $\mathrm{e}_{2}s_{3}s_{2}(T)\neq0$ because $\mathrm{e}_{3}(T)\neq
0$.\ Set $U=s_{3}s_{2}(T)$ and $U'=s_{3}s_{2}(T')$. Then
$\mathrm{e}_{2}(U')=0$ and $\mathrm{e}_{2}(U)\neq0$.\ Since
$T=\mathrm{f}_{1}(T')$ we can write
\[
U=s_{3}s_{2}\mathrm{f}_{1}s_{2}s_{3}(U')\Longleftrightarrow
U=s_{3}\mathrm{f}_{\varepsilon_{1}-\varepsilon_{3}}s_{3}(U')\,.
\]
Moreover
\[
\mathrm{wt}(U')=(\mu_{1}+1,\mu_{3},\mu_{4},\mu_{2}-1)\text{ and
}\mathrm{wt}(s_{3}(U'))=(\mu_{1}+1,\mu_{3},\mu_{2}-1,\mu_{4})\,.
\]

We shall need the two integers $a=m_{2}^{(2)}(U')$ (thus $a\leq\mu
_{3}$ for $U'$ contains $\mu_{3}$ letters $2$) and $b=m_{4}^{(1)}%
(T)$.\ We get $m_{3}^{(1)}(U')=\mu_{4}$ and $a\geq\mu_{4}$ by
Corollary~\ref{cor_e_2}. In particular, there is no letter $3$ in the second
row of $U'$. We shall discuss three cases.

{\bf Case 1:} $m_{4}^{(2)}(U')\leq m_{3}^{(1)}(U')=\mu
_{4}$.\ During the computation of $s_{3}(U')$, all the letters $4$ in
the second row of $U'$ are paired with letters $3$ of its first
row.\ Since $\mu_{2}-1-\mu_{4}\geq0$ because $\mu_{2}>\mu_{3}\geq\mu_{4}$,
exactly $\mu_{2}-1-\mu_{4}$ letters $4$ of the first row are changed in
letters $3$.\ We get $m_{3}^{(1)}(s_{3}(U'))=\mu_{2}-1$ and
$m_{2}^{(2)}(s_{3}(U'))=m_{2}^{(2)}(U')=a$.\ Thus $m_{2}%
^{(2)}(s_{3}(U'))\leq m_{3}^{(1)}(s_{3}(U'))$ (otherwise
$\mu_{3}\geq a\geq\mu_{2}$). By Lemma ~\ref{f_1,3)}, $\mathrm{f}%
_{\varepsilon_{1}-\varepsilon_{3}}(s_{3}(U'))$ is thus obtained by
changing a letter $1$ in a letter $3$ in the first row of $s_{3}(U')$.
We get $U$ by applying $s_{3}$ to $\mathrm{f}_{\varepsilon_{1}-\varepsilon
_{3}}(s_{3}(U'))$, that is by changing $\mu_{2}-\mu_{4}$ letters $3$
in letters $4$ in its first row. We see that $U$ is obtained from $U'$
by changing a letter $1$ in a letter $4$ in its first row (up to reordering).
Therefore $U$ has no letter $2$ in its second row and $m_{2}^{(2)}(U)=a\geq
\mu_{4}=m_{3}^{(1)}(U)$. By Corollary~\ref{cor_e_2} we derive the
contradiction $\mathrm{e}_{2}(U)=0$.

{\bf Case 2:} $m_{4}^{(2)}(U')>m_{3}^{(1)}(U')=\mu_{4}$.
During the computation of $s_{3}(U')$, all the letters $3$ in the
first row of $U'$ are paired with letters $4$ of its second row. Then,
all the remaining $\mu_{2}-1-\mu_{4}$ letters $4$ are changed into letters $3$
in both rows. In particular, we get $m_{3}^{(1)}(s_{3}(U'))=\mu_{4}%
+b$. We shall consider two subcases.

{\bf Case 2.a:} $a=m_{2}^{(2)}(U')\leq m_{3}^{(1)}(s_{3}%
(U'))=\mu_{4}+b$. By Lemma~\ref{f_1,3)}, $\mathrm{f}_{\varepsilon
_{1}-\varepsilon_{3}}(s_{3}(U'))$ is then obtained by changing a
letter $1$ in a letter $3$ in the first row of $s_{3}(U')$ and we get
a contraction exactly as in Case 1.

{\bf Case 2.b:} $a=m_{2}^{(2)}(U')>m_{3}^{(1)}(s_{3}(U^{\prime
}))=\mu_{4}+b$. By Lemma~\ref{f_1,3)}, $\mathrm{f}_{\varepsilon
_{1}-\varepsilon_{3}}(s_{3}(U'))$ is then obtained from $s_{3}%
(U')$ by changing a letter $1$ in a letter $2$ in its first row and a
letter $2$ in a letter $3$ in its second row. We get $U$ by applying $s_{3}$
to $\mathrm{f}_{\varepsilon_{1}-\varepsilon_{3}}(s_{3}(U'))$, that is
by changing in letters $4$ all the letters $3$ unpaired with the letters $4$
located in the second row. We see that $U$ is obtained from $U'$ by
changing a letter $1$ in a letter $2$ in its first row and a letter $2$ in a
letter $4$ in its second row.\ Thus, there is no letter $3$ in the second row
of $U$. Moreover $m_{2}^{(2)}(U)=m_{2}^{(2)}(U')-1=a-1$.\ But by
hypothesis, $a>\mu_{4}+b,$ thus $a-1\geq\mu_{4}$ and $m_{2}^{(2)}(U)\geq
\mu_{4}=m_{3}^{(1)}(U)$. We yet get the contradiction $\mathrm{e}_{2}(U)=0$ by
Corollary~\ref{cor_e_2}.
\end{proof}

\bigskip

We can now prove Lemma~\ref{acte13} by induction on the cocharge.

\begin{proof}
Observe first Lemma~\ref{acte13} is clearly true for row tableaux, that is for
tableaux of cocharge $0$. Assume it holds for any tableau of cocharge $k$ and
the tableau $T$ considered as cocharge $k+1$.\ If $T$ has two rows we are done
by the previous lemma. So we can assume that $T$ has $3$ rows. Set
$T'=\mathrm{e}_{3}(T)$. By Corollary \ref{Cor_fc=cf}, we can apply our
induction hypothesis to the tableau $C(T)$, we must have $\mathrm{e}%
_{1}(C(T'))\neq0$.\ We claim this implies that $\mathrm{e}%
_{1}(T')\neq0$. Indeed, the conditions $\mathrm{e}_{1}(C(T'))\neq0$ but $\mathrm{e}_{1}(T')=0$ would imply that the letter of
$T'$ which is used in the cyclage operation is a $2$ (any other letter would
not modify the locations of the letters $1$ and $2$ in $T'$).\ But by
definition of the cyclage, $2$ is then the leftmost letter of the shortest row
of $T'$, which is only possible if $T'$ has two rows contrary to the assumption
we made. To terminate the proof it suffices to apply Lemma~\ref{actf13} to
$T^{\prime\prime}=\mathrm{e}_{1}(T')=\mathrm{e}_{1}\mathrm{e}_{3}(T)$.
\end{proof}

\begin{remarks}

{\rm (1) One can use a similar method to give alternative proofs of
Lemmas~{\ref{actf12}}~and~\ref{acte12}, that is, without making
explicit the action of $\mathrm{f}_{2}$ and $\mathrm{f}_{\alpha_{1}+\alpha
_{2}}$. Note that only the basic properties of cyclage mentioned in Section~\ref{sec:cyccharge} are used to carry out the induction step, and no related combinatorics. }

{\rm (2) On another hand, it would be interesting to prove Lemmas~\ref{actf13}~and~\ref{acte13} in a similar way to Lemmas~{\ref{actf12}}~and~\ref{acte12}, that is, by making explicit the action of $\mathrm{f}_{3}$. While such a proof would only use the crystal structure (without referring to cyclage), we found it challenging because the Weyl group action is not easy to express in any of the combinatorial models we considered, so we were led to an unmanageable number of cases. Nevertheless, the special case considered in the proof of Lemma~\ref{lem:tworow} is manageable with the tableau model.}
\end{remarks}

\subsection{Additional modified crystal operators in type $B_n$}

In type $B_n$, in addition to the operators $\mathrm{f}_\alpha$ for $\alpha\in W\alpha_1$, i.e., a long root, we need such operators indexed by short roots. They are defined completely similarly to~\eqref{defmodcr}. Namely, given a short root $\alpha$, consider {\em any} $w\in W$ satisfying $w(\alpha_n)=\alpha$. We then define
\begin{equation}\label{defmodcr1}\mathrm{f}_{\alpha}:=w\tilde{f}_{n}w^{-1}\,,\;\;\;\;\;\mathrm{e}_{\alpha}:=w\tilde{e}_{n}w^{-1}\,.\end{equation}
Note that, in this case, the definition does not depend on the choice of $w$, because the stabilizer of $\alpha_n$ is $W_{\alpha_n}=\langle s_1,\ldots,s_{n-2},s_{\varepsilon_{n-1}}\rangle$, so all its elements commute with $\tilde{f}_{n}$; see Remark~\ref{fv}. 

All the basic properties of the modified crystal operators in Section~\ref{sec:modcr} extend to the additional operators. Moreover, to the modified crystal graph $\mathbb{B}(\lambda)$ constructed before, we add the extra edges $b\overset{\alpha}{\dashrightarrow
}b'$ when $b'=\mathrm{f}_{\alpha}(b)$ and $\alpha$ is a short positive root. 

The following properties are needed for the additional operators.

\begin{theorem}\label{Th_comm_op_B} \hfill \begin{enumerate}
\item Given $i<j<k$, assume that $\mathrm{f}_{\varepsilon_{i}-\varepsilon_j}(b)\ne0$ and $\mathrm{f}_{\varepsilon_k}(b)\ne0$. Then we have
\[\mathrm{f}_{\varepsilon_{i}-\varepsilon_j}\mathrm{f}_{\varepsilon_k}(b)=\mathrm{f}_{\varepsilon_k}\mathrm{f}_{\varepsilon_{i}-\varepsilon_j}(b)\ne0\,.\]
The same is true with the $\mathrm{f}_\cdot$ operators replaced by the $\mathrm{e}_\cdot$ operators.
\item Consider a vertex $b$ with $\langle\mathrm{wt}(b),\varepsilon_{j-1}\rangle=1$ and $\langle\mathrm{wt}(b),\varepsilon_{j}\rangle=0$. Then we have
\[\mathrm{f}_{\varepsilon_j}\mathrm{f}_{\varepsilon_{j-1}-\varepsilon_j}(b)=\mathrm{f}_{\varepsilon_{j-1}}(b)\ne0\,.\]
\item Assume that, in addition to the conditions in {\rm (1)}, we have $i<j-1$ and $\langle\mathrm{wt}(b),\varepsilon_{i}\rangle>1$. Then we have
\[\mathrm{f}_{\varepsilon_j}\mathrm{f}_{\varepsilon_i-\varepsilon_j}(b)=\mathrm{f}_{\varepsilon_i-\varepsilon_{j-1}}\mathrm{f}_{\varepsilon_{j-1}}(b)\ne0\,.
\footnote{Observe we will not need analogues of assertions (2) and (3) for the $\mathrm{e}_\cdot$ operators in the sequel.}\]
\end{enumerate}
\end{theorem}

\begin{proof} The first commutation in (1) is easy to check based on the fact that all Kashiwara crystal operators used in the definition of $\mathrm{f}_{\varepsilon_{i}-\varepsilon_j}$ commute with those corresponding to $\mathrm{f}_{\varepsilon_k}$. This is because the first operator is defined via conjugation based on $\tilde{f}_1$, whereas $\tilde{f}_n$ is used for the second. The situation is identical for the $\mathrm{e}_\cdot$ operators.

The fact that none of the expressions in (2) and (3) is $0$ holds by Lemma~\ref{fenonzero}~(1).

Let us now turn to (2). By the usual reduction procedure based on conjugation, which was applied several times in the proof of Lemma~\ref{PropRed}, the relation can be reduced to the special case corresponding to $j=2$. (In fact, the reduction procedure is simpler because the definition of the additional operators is independent of the corresponding elements of $W$.) Now observe that, since $s_i\ldots s_{n-1}(\alpha_n)=\alpha_i$, and due to the weight condition on $b$, we have 
\begin{align*}&\mathrm{f}_{\varepsilon_{2}}(b')=s_2\ldots s_{n-1}\tilde{f}_ns_{n-1}\ldots s_2(b')\,,\\&\mathrm{f}_{\varepsilon_{1}}(b)=s_1\ldots s_{n-1}\tilde{f}_ns_{n-1}\ldots s_{1}(b)=s_2\ldots s_{n-1}\tilde{f}_ns_{n-1}\ldots s_2\tilde{f}_{1}(b)\,,\end{align*}
where $b':=\mathrm{f}_{\varepsilon_1-\varepsilon_2}(b)$. Therefore, the relation to prove is reduced to $\tilde{f}_1(b)=\mathrm{f}_{\varepsilon_1-\varepsilon_2}(b)$, which is true by definition.

In order to prove (3), we plug the following identity (based on Lemma~\ref{actf12}) into the relation to prove: $\mathrm{f}_{\varepsilon_i-\varepsilon_j}(b)=\mathrm{f}_{\varepsilon_i-\varepsilon_{j-1}}\mathrm{f}_{\varepsilon_{j-1}-\varepsilon_j}(b)$. Using the commutation of $\mathrm{f}_{\varepsilon_j}$ and $\mathrm{f}_{\varepsilon_i-\varepsilon_{j-1}}$, which is true by~(1), the relation to prove becomes precisely the one in (2).
\end{proof}

\section{Proof of the atomic decomposition}\label{sec:proofatdec}

Fix a dominant weight $\lambda$ for a classical Lie algebra. Consider the subgraph of $\mathbb{B}%
(\lambda)$ consisting of the vertices of dominant weight, and the edges $b\overset{\alpha}{\dashrightarrow}\mathrm{f}_{\alpha}(b)$ for which $\mathrm{wt}(b)\gtrdot\mathrm{wt}(\mathrm{f}_{\alpha}(b))$ is a cocover in the dominant weight poset. This new colored directed graph on  the vertices of $B(\lambda)^+$ will be denoted by $\mathbb{B}(\lambda)^+$. It can also be viewed as a poset (with cocovers given by the above edges), and the weight function is a poset projection to the interval $[\widehat{0},\lambda]$ in the dominant weight poset. The two points of view will be used interchangeably.

The main goal is to identify situations in which the components of the poset $\mathbb{B}(\lambda)^+$ define an atomic, respectively $t$-atomic decomposition, cf. Definitions~\ref{crat} and~\ref{crtat}. 

\begin{remark}{\rm It is important to use the setup mentioned above, as we found that several other variations fail, as explained below.
\begin{itemize}
\item If we consider all vertices of $B(\lambda)$, rather than just those of dominant weight, then Lemma~\ref{actf13} fails, for instance, in type $A_3$, for $\lambda=(4,1,1)$.
\item In type $A_{n-1}$, we obtain the same results by defining the modified crystal operators based on ${\mathrm f}_{n-1}$ rather than ${\mathrm f}_{1}$ (see Section~\ref{sec:modcr}), due to the symmetry of the Dynkin diagram. However, if any other node of the Dynkin diagram is used, the connected components of the corresponding $\mathbb{B}(\lambda)$ do not satisfy the properties in Theorem~\ref{Th_tAtomic}, which are needed for the atomic decomposition. An example is $\lambda=(4,3)$ in type $A_3$. 
\item The same complication arises if we do not remove the modified crystal edges which do not correspond to cocovers in the dominant weight poset. An example is $\lambda=(3,1,1)$ in type $C_3$.
\end{itemize}
}
\end{remark}

Lemma~\ref{fenonzero}~(1) immediately gives the following result, which is a converse of the fact that every edge in $\mathbb{B}(\lambda)^+$ projects to a cocover in the dominant weight poset (by definition).

\begin{lemma}\label{edgeexist} Given a vertex $b$ in $\mathbb{B}(\lambda)^+$ of weight $\mu$ and a cocover $\mu\gtrdot\mu-\alpha$ in the dominant weight poset (with $\alpha\in R^+$), we have an edge $b\overset{\alpha}{\dashrightarrow}\mathrm{f}_{\alpha}(b)$ in $\mathbb{B}(\lambda)^+$.
\end{lemma}

\subsection{Type $A_{n-1}$}\label{main-a}

We refer freely to Section~\ref{dwpa}. 

\begin{lemma}\label{lemdn} Consider two distinct edges $b{\dashrightarrow}b'$ and $b{\dashrightarrow}b''$ in $\mathbb{B}(\lambda)^+$. The vertices $b'$ and $b''$ have a lower bound in this poset. 
\end{lemma}

\begin{proof} Let $\mu:=\mathrm{wt}(b)$, $\nu:=\mathrm{wt}(b')$, and $\pi:=\mathrm{wt}(b'')$. We have cocovers $\mu\gtrdot\nu$ and $\mu\gtrdot\pi$, so the interval $[\nu\wedge\pi,\mu]$ has one of the structures in Cases A$1-$A$3$. Starting from the crystal vertex $b$, we can apply the $\mathrm{f}_\cdot$ operators indexed by the labels in the corresponding diagrams, by Lemma~\ref{edgeexist}. It suffices to check that these diagrams commute, which follows by using Theorem~\ref{Th_comm_op}~(1)-(2) repeatedly. In fact, we apply this theorem to the corresponding triangles and diamonds, after we verify its hypotheses by inspecting the diagrams. In the case of the pentagons, for instance in Case~A2~(a), we let $\mathrm{f}_{ij}:=\mathrm{f}_{\alpha_{ij}}$, and calculate:
\[\mathrm{f}_{i+2,j}\mathrm{f}_{i+1,i+2}(\mathrm{f}_{i,i+1}(b))=\mathrm{f}_{i+1,j}\mathrm{f}_{i,i+1}(b)=\mathrm{f}_{i,i+1}\mathrm{f}_{i+1,j}(b)\,;\]
indeed, each of the two equalities follows from Theorem~\ref{Th_comm_op}~(1).
\end{proof}

\begin{lemma}\label{lemup} Consider two distinct edges $b'{\dashrightarrow}b$ and $b''{\dashrightarrow}b$ in $\mathbb{B}(\lambda)^+$. The vertices $b'$ and $b''$ have an upper bound in this poset.
\end{lemma}

\begin{proof} The notation and related conditions are the same as in Section~\ref{dwpa}. Like in the proof of Lemma~\ref{lemdn}, the goal is to lift the diagrams in Cases A$1'-$A$3'$ from the dominant weight poset to $\mathbb{B}(\lambda)^+$. In fact, we will perform the lift along the edges of slightly modified diagrams (by starting from the bottom) in the cases indicated below, while we use the same diagrams as in Section~\ref{dwpa} in Cases A$1'$ (a) and A$1'$ (b).

{\bf Case A1$\mathbf{'}$ (c).}
\[\scriptstyle
{\xymatrix{
&{\ldots ab^pc^qd\ldots}\ar[ddl]_{(i,j)}\ar[ddr]^{(j-1,k)}\ar[d]_{(j-1,j)}\\
&{\;\;\;\;\;\;\;\ldots ab^{p-1}cbc^{q-1}d\ldots\;\;\;\;\;\;\;}\ar[dl]^{(i,j-1)}\ar[dr]_{(j,k)}\\
{\;\;\;\;\;\;\;\ldots b^{p+2}c^{q-1}d\ldots\;\;\;\;\;\;\;}\ar[dr]_{(j,k)}\ar@<0.2mm>[dr]\ar@<0.4mm>[dr]&&{\;\;\;\;\;\;\;\ldots ab^{p-1}c^{q+2}\ldots\;\;\;\;\;\;\;}\ar[dl]^{(i,j-1)}\ar@<0.2mm>[dl]\ar@<0.4mm>[dl]\\
&{\ldots b^{p+1}c^{q+1}\ldots}
} }
\]

{\bf Case A2$\mathbf{'}$ (a).}
\[\scriptstyle
{\xymatrix{
&{\ldots acd^pe\ldots}\ar[dl]_{(i,i+1)}\ar[d]_{(i+1,i+2)}\ar[ddr]^{(i+1,j)}\\
{\ldots(a-1)(c+1)d^pe\ldots}\ar[d]_{(i+1,i+2)}&{\ldots adcd^{p-1}e\ldots}\ar[dl]^{(i,i+1)}\ar[dr]_{(i+2,j)}\\
{\ldots(a-1)c^2d^{p-1}e\ldots}\ar[dr]_{(i+2,j)}\ar@<0.2mm>[dr]\ar@<0.4mm>[dr]&&{\ldots ad^{p+2}\ldots}\ar[dl]^{(i,i+1)}\ar@<0.2mm>[dl]\ar@<0.4mm>[dl]\\
&{\ldots(a-1)cd^{p+1}\ldots}} }
\]

{\bf Case A2$\mathbf{'}$ (b)} is similar to Case A$2'$ (a).

{\bf Case A3$\mathbf{'}$.}
\[\scriptstyle
{\xymatrix{
&{\ldots (a+1)cd(f-1)\ldots } \ar[d]_{(i+1,i+2)} \ar[dl]_{(i+2,i+3)} \ar[dr]^{(i,i+1)} \\
{\ldots (a+1)c(d-1)f \ldots } \ar[d]_{(i+1,i+2)}& {\ldots (a+1)dc(f-1)\ldots}\ar[dl]^{(i+2,i+3)}\ar[dr]_{(i,i+1)}&{\ldots a(c+1)d(f-1)\ldots} \ar[d]^{(i+1,i+2)}  \\
{\ldots (a+1)d^2f\ldots }\ar[dr]_{(i,i+1)}\ar@<0.2mm>[dr]\ar@<0.4mm>[dr] & & {\ldots ac^2(f-1)\ldots}\ar[dl]^{(i+2,i+3)}\ar@<0.2mm>[dl]\ar@<0.4mm>[dl]\\
&{\ldots acdf\ldots }
} }
\]

In Cases A$1'$ (a) and A$1'$ (b), the lemma is a direct consequence of Theorem~\ref{Th_comm_ope}~(2) and~(1), respectively. Note that, in the first case, the extra condition in Theorem~\ref{Th_comm_ope}~(2) on the dominant weight $\ldots a\ldots c\ldots e\ldots g\ldots$ at the bottom of the diagram amounts to $c>e$, where the two covers of the mentioned weight are $\ldots a+1\ldots c-1\ldots e\ldots g\ldots$ and $\ldots a\ldots c\ldots e+1\ldots g-1\ldots$; this condition immediately follows from the fact that the covers are dominant weights themselves.  

Thus, it suffices to focus on the three diagrams above. Their distinctive feature is that the weight in the middle is not dominant, but a single pair of consecutive entries is in increasing order. In each of the three cases, we start by applying Theorem~\ref{Th_comm_ope}~(2) to the diamond at the bottom (the extra condition in the theorem is part of the assumptions corresponding to the mentioned cases). Then we can apply the corresponding $\mathrm{e}_\cdot$ operator to the determined vertex of nondominant weight, by Lemma~\ref{fenonzero}~(2). Finally, starting from the determined top vertex, we can apply the corresponding $\mathrm{f}_\cdot$ operators along all the remaining edges of the modified diagrams, by Lemma~\ref{edgeexist}. The commutativity of the remaining triangles and diamonds is checked by Theorem~\ref{Th_comm_op}~(1); indeed, in each case we verify the hypothesis by inspecting the corresponding diagram. This concludes the proof.
\end{proof}

\begin{theorem}\label{Th_tAtomic}The components of $\mathbb{B}(\lambda)^+$ define a $t$-atomic decomposition. Moreover, these components are isomorphic to intervals of the form $[\widehat{0},\mu]$ in the dominant weight poset via the weight projection, and the distinguished vertex $h\in H(\lambda)$ in each of them is chosen to be the corresponding maximum. 
\end{theorem}

\begin{proof} By Lemma~\ref{edgeexist}, the weight projection of a lower order ideal determined by a vertex of weight $\mu$ is the interval $[\widehat{0}, \mu]$. 

Now fix a weakly connected component $C$ of $\mathbb{B}(\lambda)^+$. As we saw, all of its minimal vertices have weight $\widehat{0}$. We will first prove the uniqueness of a minimal vertex. Assuming the contrary, let $\overline{C}$ be the subposet of $C$ consisting of vertices connected via directed paths to more than one minimal vertex. Find an undirected path in $C$ connecting two distinct minimal vertices, and consider its local maxima and minima which are not endpoints. As noted above, there are directed paths from the local minima to minimal vertices. By considering these paths, we can see that some local maximum must be in $\overline{C}$, so $\overline{C}\ne\emptyset$; see Figure~\ref{fig1}. Fixing a minimal vertex $b$ in $\overline{C}$, we must have distinct edges  $b{\dashrightarrow}b'$ and $b{\dashrightarrow}b''$. 
By Lemma~\ref{lemdn}, $b'$ and $b''$ have a lower bound $\overline{b}$ in $C$. By considering a directed path from $\overline{b}$ to a minimal vertex, we can see that $b'$ or $b''$ are in $\overline{C}$, which contradicts the minimality of $b$; see Figure~\ref{fig2}. Therefore, $C$ has a minimum $b_{\min}$. The existence of a maximum $b_{\max}$ is proved in a completely similar way, by using Lemma~\ref{lemup} instead. 
\begin{figure}
\begin{minipage}[b]{0.44\textwidth}   
    \includegraphics[width=\textwidth]{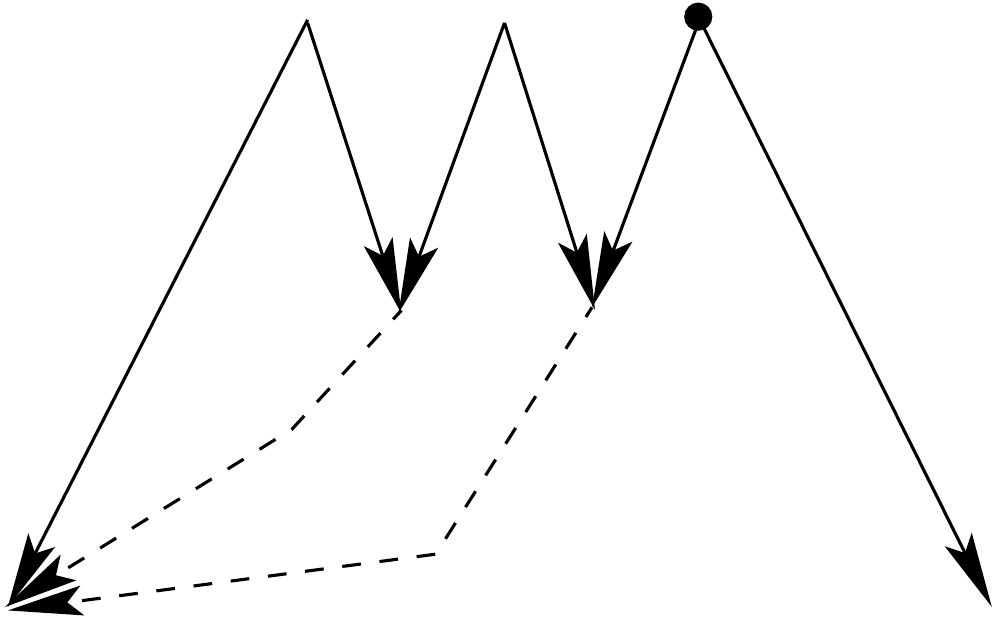}
		\caption{Proof of Theorem~\ref{Th_tAtomic}: $\overline{C}\ne\emptyset$.}\label{fig1}
\end{minipage}
\quad\quad\begin{minipage}[b]{0.44\textwidth}
    \includegraphics[width=\textwidth]{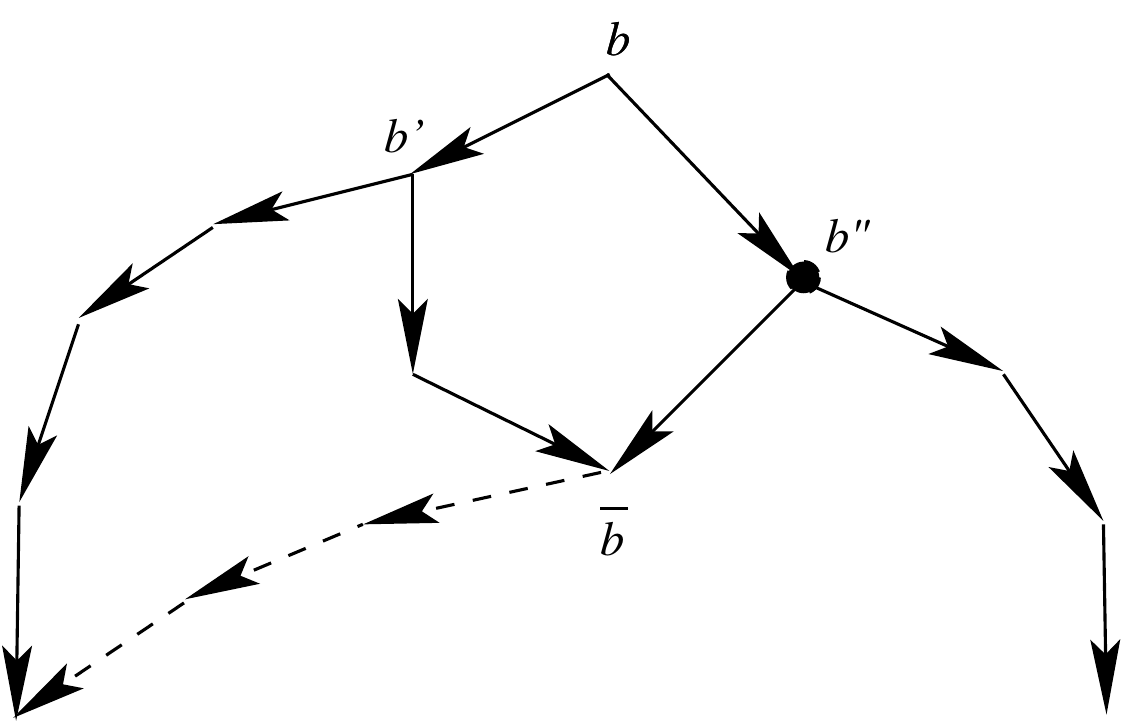}
		\caption{Proof of Theorem~\ref{Th_tAtomic}: $C$ has a minimum.}\label{fig2}
\end{minipage}
\end{figure}

We then need to show that there are no two vertices of the same weight in a component. Assume for contradiction that $b$ and $b'$ are such vertices, with $\mathrm{wt}(b)=\mathrm{wt}(b')=\mu$. Then we can find a saturated decreasing chain from $\mu$ to $\widehat{0}$ in the dominant weight poset. According to Lemma~\ref{edgeexist}, by applying to $b$ and $b'$ the $\mathrm{f}_{\cdot}$ operators corresponding to the labels on the mentioned chain, we obtain directed paths from these vertices to the minimum $b_{\min}$ in the considered component. However, using the reverse sequence of labels and starting from $b_{\min}$, it is clearly impossible to reach two different vertices via the $\mathrm{e}_{\cdot}$ operators.

It is now clear that the components of $\mathbb{B}(\lambda)^+$ define an atomic decomposition, where in each component we choose its maximum as the distinguished vertex $h\in H(\lambda)$.

To get the $t$-atomic decomposition, we first need, according to Definition~\ref{defat}, a
statistic on $H(\lambda)$. One can use the realization of $B(\lambda)$ in
terms of semistandard tableaux. Obviously, the natural candidate for the statistic $\mathrm{c}(\cdot)$ in Definition~\ref{crtat} is the Lascoux-Sch\"utzenberger charge~\cite{LSc1}. For any dominant weight $\mu$ (i.e. for any
partition with at most $n$ parts), we set%
\[
A_{\lambda,\mu}=\sum_{\substack{h\in H(\lambda) \\ \mathrm{wt}(h)=\mu}}t^{\mathrm{c}(h)}\,.
\]
We then have for any dominant weight $\nu$%
\[
K_{\lambda,\nu}(t)=\sum_{T\in B(\lambda)_{\nu}}t^{\mathrm{c}(T)}=\sum_{\nu
\leq\mu\leq\lambda}\sum_{\substack{h\in H(\lambda)\\ \mathrm{wt}(h)=\mu}%
}\sum_{T\in B(\lambda)_{\nu}\cap\mathbb{B}(\lambda,h)}t^{\mathrm{c}(T)}\,.
\]
Now by Proposition \ref{Prop_Modop_charge}, we obtain the equality
$\mathrm{c}(T)=\mathrm{c}(h)+\langle\mu-\nu,\rho^{\vee}\rangle$ for any $T\in
B(\lambda)_{\nu}\cap\mathbb{B}(\lambda,h)$.$\ $Indeed we have $\mathrm{wt}%
(h)-\mathrm{wt}(T)=\mu-\nu$ and $T$ can be obtained from $h$ by applying
modified crystal operators $\mathrm{f}_{\alpha},$ each of them increasing the
charge by $\langle\alpha,\rho^{\vee}\rangle$.\ Also, the set $B(\lambda)_{\nu
}\cap\mathbb{B}(\lambda,h)$ is reduced to a singleton because the connected
component $\mathbb{B}(\lambda,h)$ of $\mathbb{B}(\lambda)^{+}$ contains
exactly one vertex of weight $\nu\leq\lambda$. Thus we can write%
\[
K_{\lambda,\nu}(t)=\sum_{\nu\leq\mu\leq\lambda}t^{\langle\mu-\nu,\rho^{\vee
}\rangle}\sum_{\substack{h\in H(\lambda)\\ \mathrm{wt}(h)=\mu}}t^{\mathrm{c}%
(h)}=\sum_{\nu\leq\mu\leq\lambda}t^{\langle\mu-\nu,\rho^{\vee}\rangle
}A_{\lambda,\mu}(t)\,,
\]
which is equivalent to the $t$-atomic decomposition by Proposition~\ref{defateq}.
\end{proof}

\begin{example}\label{exa3} {\rm Consider $\lambda=(3,2,1)$ in type $A_3$. The modified crystal graph $\mathbb{B}(\lambda)^+$ is shown in Figure~\ref{fig3}. Its vertices are labeled by semistandard Young tableaux whose content is a partition, and its edges are labeled as above. In particular, this graph gives the following atomic decomposition of the character:
\[\chi_\lambda=w_{(3,2,1)}+w_{(2,2,2)}+w_{(3,1,1,1)}+w_{(2,2,1,1)}\,.\]}
\begin{figure}
\[\scriptstyle
 {\xymatrix{
&{\tableau{1&1&1\\2&2\\3}}\ar[dl]_{(1,3)}\ar[dr]^{(2,4)}
&&{\tableau{1&1&3\\2&2\\3}}\ar[d]_{(3,4)}
&{\tableau{1&1&1\\2&4\\3}}\ar[d]_{(1,2)}
&{\tableau{1&1&4\\2&2\\3}}
\\
{\tableau{1&1&2\\2&3\\3}}\ar[dr]_{(3,4)}&&{\tableau{1&1&1\\2&3\\4}}\ar[dl]^{(1,2)}
&{\tableau{1&1&3\\2&2\\4}}
&{\tableau{1&1&2\\2&4\\3}}
\\
&{\tableau{1&1&2\\2&3\\4}}
}
}\]
		\caption{The modified crystal graph $\mathbb{B}(\lambda)^+$ in Example~{\rm \ref{exa3}}.}\label{fig3}
\end{figure}
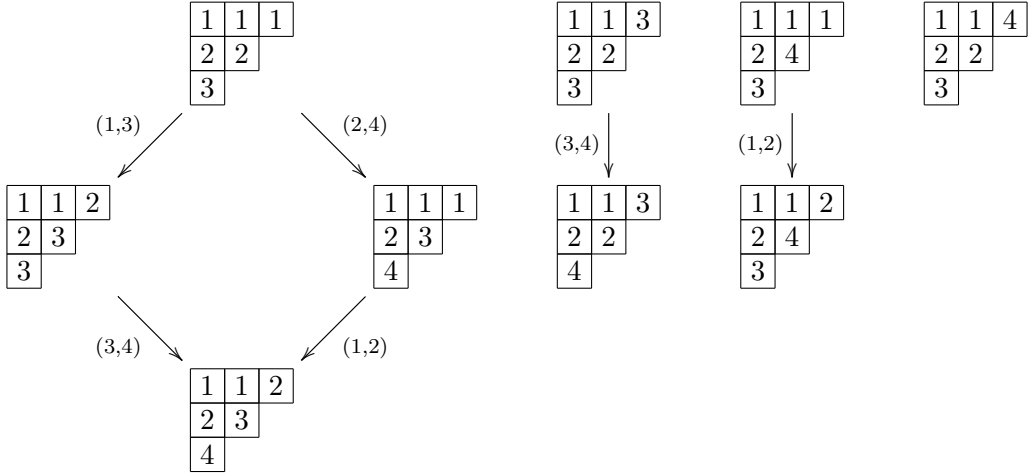
\end{example}


\subsection{Types $B_n$, $C_{n}$, and $D_n$}

This section refers to the stable ranges in types $B_n$, $C_{n}$, and $D_n$, namely to a corresponding graph/poset $\mathbb{B}(\lambda)^+$. We refer freely to  Sections~\ref{domwc},~\ref{domwd}, and~\ref{domwb}, as well as to the results in Section~\ref{main-a}. 

\begin{lemma}\label{lemdnup} Lemmas {\rm \ref{lemdn}} and {\rm \ref{lemup}} hold in types $B_n$, $C_{n}$, and $D_n$, in the corresponding stable ranges.
\end{lemma}

\begin{proof} We start with types $C_n$ and $D_n$. In addition to the cases considered in the proofs of Lemmas~{\rm \ref{lemdn}} and {\rm \ref{lemup}}, we need to consider the cases involving the new type of cover in the corresponding dominant weight poset, namely Cases~C$1-$C3 (for extending the first lemma, by Proposition~\ref{intcdn}) and Cases C$1'-$C$2'$ (for extending the second one, by Proposition~\ref{intcup}). The goal is the same: to lift the corresponding diagrams from the dominant weight poset to $\mathbb{B}(\lambda)^+$. We use the same reasoning as in the mentioned lemmas, and give more details below.

For instance, in order to prove Lemma~\ref{lemdn} in Case~C2, we let $\mathrm{f}_{i\overline{\jmath}}:=\mathrm{f}_{\alpha_{i\overline{\jmath}}}$, and calculate:
\[\mathrm{f}_{i+2,\overline{i+3}}\mathrm{f}_{i+1,i+3}(\mathrm{f}_{i,i+1}(b))=\mathrm{f}_{i+1,\overline{i+2}}\mathrm{f}_{i,i+1}(b)=\mathrm{f}_{i,i+1}\mathrm{f}_{i+1,\overline{i+2}}(b)\,;\]
indeed, each of the two equalities follows from Theorem~\ref{Th_comm_op}~(1) after we verify its hypothesis by inspecting the corresponding diagram.

We now turn to Lemma~\ref{lemup}. Like in the proof of its type $A$ version, we will perform the lift along the edges of slightly modified diagrams (by starting from the bottom) in the cases indicated below, while we use the same diagram as in Section~\ref{domwc} in Case C$1'$~(a). The reasoning is completely similar, based on repeatedly applying Theorems~\ref{Th_comm_ope}~(2) and \ref{Th_comm_op}~(1), after carefully verifying their hypotheses each time. Note that the special condition in Theorem~\ref{Th_comm_ope}~(2) requires $n>4$ in type $D_n$, but this is clearly true in the stable range. 

{\bf Case C1$\mathbf{'}$ (b).}
\[\scriptstyle
{\xymatrix{
&{\ldots 21^p}\ar[ddl]_{(i,j)}\ar[ddr]^{(j-2,\overline{j-1})}\ar[d]_{(j-2,j)}\\
&{\;\;\;\;\;\;\;\;\;\;\;\;\ldots 21^{p-2}01^2\;\;\;\;\;\;\;\;\;\;\;\;}\ar[dl]^{(i,j-2)}\ar[dr]_{(j-1,\overline{\jmath})}\\
{\;\;\;\;\;\;\;\;\;\;\;\;\ldots 1^{p+2}\;\;\;\;\;\;\;\;\;\;\;\;}\ar[dr]_{(j-1,\overline{\jmath})}\ar@<0.2mm>[dr]\ar@<0.4mm>[dr]&&{\;\;\;\;\;\;\;\;\;\;\;\;\ldots 21^{p-2}\;\;\;\;\;\;\;\;\;\;\;\;}\ar[dl]^{(i,j-2)}\ar@<0.2mm>[dl]\ar@<0.4mm>[dl]\\
&{\ldots 1^p}
} }
\]

{\bf Case C2$\mathbf{'}$.}
\[\scriptstyle
{\xymatrix{
&{\ldots a1^2}\ar[dl]_{(i,i+1)}\ar[d]_{(i+1,i+3)}\ar[ddr]^{(i+1,\overline{i+2})}\\
{\;\;\;\;\;\;\;\ldots(a-1)21\;\;\;\;\;\;\;}\ar[d]_{(i+1,i+3)}&{\ldots a01^2}\ar[dl]^{(i,i+1)}\ar[dr]_{(i+2,\overline{i+3})}\\
{\;\;\;\;\;\;\;\ldots(a-1)1^3\;\;\;\;\;\;\;}\ar[dr]_{(i+2,\overline{i+3})}\ar@<0.2mm>[dr]\ar@<0.4mm>[dr]&&{\;\;\;\;\;\;\;\ldots a\;\;\;\;\;\;\;}\ar[dl]^{(i,i+1)}\ar@<0.2mm>[dl]\ar@<0.4mm>[dl]\\
&{\ldots(a-1)1}} }
\]

As for type $B_n$, the reasoning is completely similar, based on Theorem~\ref{Th_comm_op_B}. To be more precise, for the $\mathrm{f}_\cdot$ operators, we use Cases B1$-$B3 in Section~\ref{domwb}; for the $\mathrm{e}_\cdot$ operators, the two diagrams above (for Cases  C1$\mathbf{'}$ (b) and C2$\mathbf{'}$)  are replaced with the following ones below (for Cases  B1$\mathbf{'}$ (b) and B2$\mathbf{'}$), respectively. 

{\bf Case B1$\mathbf{'}$ (b).}
\[\scriptstyle
{\xymatrix{
&{\ldots 21^p}\ar[ddl]_{(i,j)}\ar[ddr]^{(j-1)}\ar[d]_{(j-1,j)}\\
&{\;\;\;\;\;\;\;\;\;\;\;\;\ldots 21^{p-1}01\;\;\;\;\;\;\;\;\;\;\;\;}\ar[dl]^{(i,j-1)}\ar[dr]_{(j)}\\
{\;\;\;\;\;\;\;\;\;\;\;\;\ldots 1^{p+2}\;\;\;\;\;\;\;\;\;\;\;\;}\ar[dr]_{(j)}\ar@<0.2mm>[dr]\ar@<0.4mm>[dr]&&{\;\;\;\;\;\;\;\;\;\;\;\;\ldots 21^{p-1}\;\;\;\;\;\;\;\;\;\;\;\;}\ar[dl]^{(i,j-1)}\ar@<0.2mm>[dl]\ar@<0.4mm>[dl]\\
&{\ldots 1^{p+1}}
} }
\]

{\bf Case B2$\mathbf{'}$.}
\[\scriptstyle
{\xymatrix{
&{\ldots a1}\ar[dl]_{(i,i+1)}\ar[d]_{(i+1,i+2)}\ar[ddr]^{(i+1)}\\
{\;\;\;\;\;\;\;\ldots(a-1)2\;\;\;\;\;\;\;}\ar[d]_{(i+1,i+2)}&{\ldots a01}\ar[dl]^{(i,i+1)}\ar[dr]_{(i+2)}\\
{\;\;\;\;\;\;\;\ldots(a-1)1^2\;\;\;\;\;\;\;}\ar[dr]_{(i+2)}\ar@<0.2mm>[dr]\ar@<0.4mm>[dr]&&{\;\;\;\;\;\;\;\ldots a\;\;\;\;\;\;\;}\ar[dl]^{(i,i+1)}\ar@<0.2mm>[dl]\ar@<0.4mm>[dl]\\
&{\ldots(a-1)1}} }
\]
\end{proof}

\begin{theorem}\label{Th_Atomic-C}The components of $\mathbb{B}(\lambda)^+$ define an atomic decomposition. Moreover, these components are isomorphic to intervals of the form $[\widehat{0},\mu]$ in the dominant weight poset via the weight projection. 
\end{theorem}

\begin{proof} We use the same reasoning as in the first part of the proof of Theorem~\ref{Th_tAtomic} (the one referring to the atomic decomposition, as opposed to the $t$-atomic decomposition). The proof is based on Lemma~\ref{lemdnup} instead.
\end{proof}

\begin{example}\label{exc3} {\rm Consider $\lambda=(2,1,1)$ in type $C_3$. The modified crystal graph $\mathbb{B}(\lambda)^+$ is shown in Figure~\ref{fig4}. Its vertices are labeled by Kashiwara-Nakashima tableaux of dominant weight, and its edges are labeled as above. In particular, this graph gives the following atomic decomposition of the character:
\[\chi_\lambda=w_{(2,1,1)}+2w_{(1,1,0)}+w_{(0,0,0)}\,.\]}
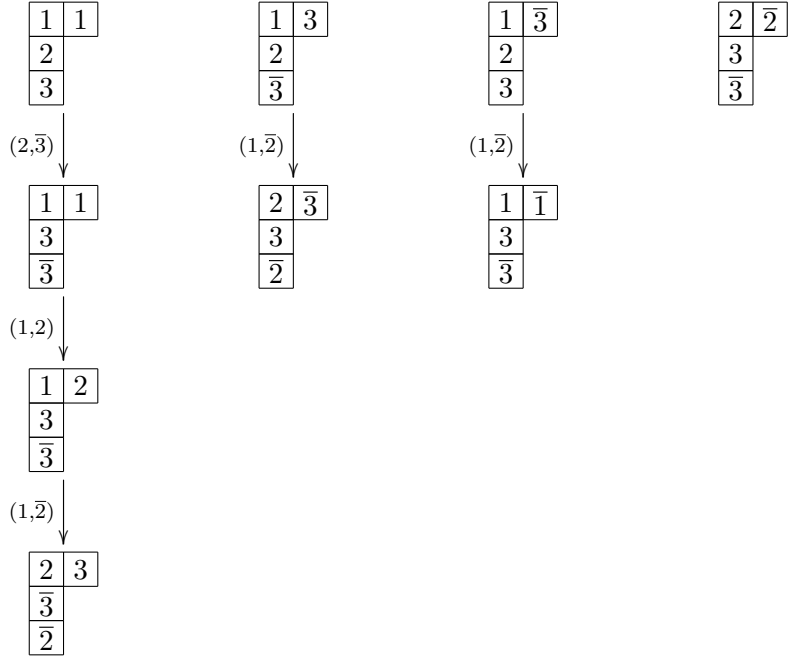
\begin{figure}
 \[\scriptstyle
{\xymatrix{
{\tableau{1&1\\2\\{{3}}}}\ar[d]_{(2,\overline{3})}
&&{\tableau{1&3\\2\\{\overline{3}}}}\ar[d]_{(1,\overline{2})}
&&{\tableau{1&{\overline{3}}\\2\\{{3}}}}\ar[d]_{(1,\overline{2})}
&&{\tableau{2&{\overline{2}}\\3\\{\overline{3}}}}
\\
{\tableau{1&{{1}}\\3\\{\overline{3}}}}\ar[d]_{(1,2)}
&&{\tableau{2&{\overline{3}}\\3\\{\overline{2}}}}
&&{\tableau{1&{\overline{1}}\\3\\{\overline{3}}}}
\\
{\tableau{1&{{2}}\\3\\{\overline{3}}}}\ar[d]_{(1,\overline{2})}\\
{\tableau{2&3\\{\overline{3}}\\{\overline{2}}}}
}
}
\]
\caption{The modified crystal graph $\mathbb{B}(\lambda)^+$ in Example~{\rm \ref{exc3}}.}\label{fig4}
\end{figure}
\end{example}

\section{Additional facts and perspectives}\label{sec:conj}

\subsection{The $t$-atomic decomposition for the adjoint representation.}

Let $\tilde{\alpha}$ be the highest root of the Lie algebra $\mathfrak{g}%
$.\ When the root system of $\mathfrak{g}$ is simply laced, $\tilde{\alpha}$
is the unique positive root which is also a dominant weight. Otherwise, there
is another positive root $\hat{\alpha}$ which is dominant, and we have $0\leq\hat
{\alpha}\leq\tilde{\alpha}$.\ More precisely, both $\hat{\alpha}$ and
$\tilde{\alpha}-\hat{\alpha}$ are short roots of $R_{+}$. In the crystal
$B(\tilde{\alpha})$, there are vertices $b_{\alpha}$ of weight $\alpha$, one
for each root $\alpha$ of $\mathfrak{g}$ and $r$ vertices of weight $0$. 

\begin{lemma}
For any simple $\alpha_{i}$ and any index $j\in\{1,\ldots,r\}$, we have
$\tilde{f}_{j}(b_{\alpha_{i}})\neq0$ if and only if $i=j$.
\end{lemma}

\begin{proof}
Since $\alpha_{i}\neq-\tilde{\alpha}$ the lowest weight in $B(\tilde{\alpha}%
)$, there is at least an index $j$ such that $f_{j}(b_{\alpha_{i}})\neq0$. If
$j\neq i$, the vertex $f_{j}(b_{\alpha_{i}})$ has weight $\alpha_{i}%
-\alpha_{j}=\alpha\in R$. When $\alpha\in R_{+}$ (resp. when $-\alpha\in
R_{+}$), we get a contradiction because $\alpha_{i}=\alpha_{j}+\alpha$ is not
simple (resp. $\alpha_{j}=\alpha_{i}+(-\alpha)$ is not simple). 
\end{proof}

It follows from the lemma that the vertices $b_{i}=\tilde{f}_{\alpha_{i}%
}(b_{\alpha_{i}})$, $i=1,\ldots,r$, are the zero weight vertices in $B(\tilde
{\alpha})$. Recall also that
\[
K_{\tilde{\alpha},0}(t)=\sum_{i=1}^{r}t^{m_{i}}\,,
\]
where $m_{1},\ldots,m_{r}$ are the classical exponents of $\mathfrak{g}$.\ We
can choose $m_{1}=\left\vert \tilde{\alpha}\right\vert $ since $\left\vert
\tilde{\alpha}\right\vert $ is the greatest exponent. 

Assume first that $\alpha_{1}$ is short, that is, the root system is not of
type $B_{r}$ or $F_{4}$. In the simply laced case, the highest root
$\tilde{\alpha}$ is in the orbit of $\alpha_{1}$ and, since $\langle
\tilde{\alpha},\tilde{\alpha}\rangle>0$, we derive by
Lemma~\ref{Lemma_f_dominant} that $\mathrm{f}_{\tilde{\alpha}}(\tilde{\alpha
})\neq0$ is a vertex of zero weight in $B(\tilde{\alpha})$.\ In fact, the
previous lemma also implies that $\mathrm{f}_{\tilde{\alpha}}(\tilde{\alpha
})=b_{1}$, because $\mathrm{f}_{\alpha}(b_{i})=\mathrm{e}_{\alpha}(b_{i})=0$
for any $i\neq1$ and any positive root $\alpha$. Indeed, each vertex $b_{i}$
has zero weight and thus is fixed under the action of the Weyl group, while 
$\tilde{f}_{1}(b_{i})\neq0$ if and only if $i=1$. In the non-simply laced case
(that is, in types $C_{r}$ and $G_{2}$ under our assumption), we get similarly
$\mathrm{f}_{\tilde{\alpha}-\hat{\alpha}}(\tilde{\alpha})=b_{\hat{\alpha}}$
(because $\langle\tilde{\alpha},\tilde{\alpha}-\hat{\alpha}\rangle>0$) and
$\mathrm{f}_{\hat{\alpha}}(\hat{\alpha})=b_{1}$. In all cases, the previous
actions of the modified operators correspond to coverings in the dominant
weight poset. Therefore, with the notation of Section~\ref{sec:tatdec}, we have
$H(\tilde{\alpha})=\{b_{\tilde{\alpha}},b_{2},\ldots,b_{r}\}$, with
$\mathbb{B}(\tilde{\alpha},b_{i})=\{b_{i}\}$ for any $i=2,\ldots,r$, and
\[
\mathbb{B}(\tilde{\alpha},\tilde{\alpha}):\left\{
\begin{array}
[c]{c}%
b_{\tilde{\alpha}}\overset{\tilde{\alpha}}{\dashrightarrow}b_{1}\text{ in the
simply laced cases}\\[1.2mm]
b_{\tilde{\alpha}}\overset{\tilde{\alpha}-\hat{\alpha}}{\dashrightarrow
}b_{\tilde{\alpha}}\overset{\hat{\alpha}}{\dashrightarrow}b_{1}\text{ in types
}C_{r}\text{ and }G_{2}\,.%
\end{array}
\right.
\]
Now define a statistic $\mathrm{c}$ on $H(\tilde{\alpha})$ such that $\mathrm{c}(b_{\tilde
{\alpha}})=0$ and $\{\mathrm{c}(b_{2}),\ldots,\mathrm{c}(b_{r})\}=\{m_{2},\ldots,m_{r}\}$. We
can then extend it to $B(\tilde{\alpha})^{+}$ by setting 
\[\mathrm{c}(b_{1}%
)=\mathrm{c}(b_{\tilde{\alpha}})+\langle\tilde{\alpha},\rho^{\vee}\rangle=\left\vert
\tilde{\alpha}\right\vert\,,\;\;\;\;\;\mbox{and}\;\;\;\;\;\mathrm{c}(b_{\hat{\alpha}})=\mathrm{c}(b_{\tilde{\alpha}%
})+\langle\tilde{\alpha}-\hat{\alpha},\rho^{\vee}\rangle=\left\vert
\tilde{\alpha}\right\vert -\left\vert \hat{\alpha}\right\vert\,.\]
 We then get
\[
A_{\tilde{\alpha},\tilde{\alpha}}(t)=\sum_{\substack{h\in H(\tilde{\alpha}%
)\\ \mathrm{wt}(h)=\tilde{\alpha}}}t^{\mathrm{c}(h)}=1\;\;\;\text{ and }\;\;\;A_{\tilde{\alpha}%
,0}(t)=\sum_{\substack{h\in H(\tilde{\alpha})\\ \mathrm{wt}(h)=0}}t^{\mathrm{c}(h)}=t^{m_{2}}%
+\cdots+t^{m_{r}}\,.
\]
For types $C_{r}$ and $G_{2}$, we also have
\[
A_{\tilde{\alpha},\hat{\alpha}}(t)=\sum_{\substack{h\in H(\tilde{\alpha})\\ \mathrm{wt}%
(h)=\hat{\alpha}}}t^{\mathrm{c}(h)}=0\,.
\]
Finally, by Proposition \ref{defateq},  we get the desired $t$-atomic decomposition
\[
K_{\tilde{\alpha},0}(t)=\left\{
\begin{array}
[c]{l}%
t^{0}A_{\tilde{\alpha},0}(t)+t^{\left\vert \tilde{\alpha}\right\vert
}A_{\tilde{\alpha},\tilde{\alpha}}(t)\text{ in the simply laced case}\\[2mm]
t^{0}A_{\tilde{\alpha},0}(t)+t^{\left\vert \tilde{\alpha}\right\vert
-\left\vert \hat{\alpha}\right\vert }A_{\tilde{\alpha},\hat{\alpha}%
}(t)+t^{\left\vert \tilde{\alpha}\right\vert }A_{\tilde{\alpha},\tilde{\alpha
}}(t)\text{ for types }C_{r}\text{ and }G_{2}\,.%
\end{array}
\right.
\]

Now it remains to consider types $B_{r}$ and $F_{4}$ where $\alpha_{1}$ is a
long root. Recall Stembridge's result \cite{stepod} stating that any simple
root gives a cover in dominant weight poset. As we have already seen for type $B_{r}$, we cannot define only the modified crystal operators based on $\tilde{f}_{1}$, because this would not
permit to get the covering relations corresponding to short roots.\ We shall also need the modified crystal operators obtained by Weyl group conjugation of the ordinary
crystal operator $\tilde{f}_{r}$ with $\alpha_{r}$ the short simple root. In fact, for the adjoint representation in type $B_{r}$ and $F_{4}$, the coverings we need only make
the short roots appear. So we only need, for any
short root $\alpha\in R_{+}$, the modified operators $\mathrm{f}_{\alpha
}=u_{\alpha}\tilde{f}_{r}u_{\alpha}^{-1}$, where $u_{\alpha}\in W$ is of minimal length such
that $u(\alpha_{r})=\alpha$. By using the same arguments as above, we then
also get a $t$-atomic decomposition for $K_{\tilde{\alpha},0}(t)$, this
time with $\mathbb{B}(\tilde{\alpha},b_{i})=\{b_{i}\}$ for any $i\neq n$
and
\[
\mathbb{B}(\tilde{\alpha},\tilde{\alpha}):b_{\tilde{\alpha}}\overset
{\tilde{\alpha}-\hat{\alpha}}{\dashrightarrow}b_{\hat{\alpha}}\overset
{\hat{\alpha}}{\dashrightarrow}b_{n}\,\text{.}%
\]

\begin{remark}{\rm 
In simply laced cases, we expect that the modified crystal operators defined from $\tilde{f}_{1}$ suffice to derive an atomic decomposition of crystals with mild assumptions on the highest weight considered (as in
Example \ref{counterex}).\ With similar restrictions and by considering also the modified crystal operators defined from $\tilde{f}_{n}$, this should also hold in the non simply laced cases even if new commutation relations will be probably needed in types $F_{4}$ and $G_{2}$.
}\end{remark}

\subsection{More about the charge}

In Section \ref{sec:tatdec}, we defined the $t$-atomic decomposition of $%
B(\lambda)$ from a charge statistic on $H(\lambda)$ which propagates on the
vertices $b$ of each component $\mathbb{B}(\lambda,h)$, $h\in H(\lambda)$,
by the formula~\eqref{propac}. %
When such an atomic decomposition exists, it yields a combinatorial
description of $K_{\lambda,\mu}(t)$ by (\ref{combkf}). In type $A$, Theorem~%
\ref{Th_tAtomic} gives a $t$-atomic decomposition of crystals, where $%
\mathrm{c}$ is the Lascoux-Sch\"{u}tzenberger charge statistic.

Conversely, if we fix $\nu\in P_{+}$ and assume that we have both a
combinatorial description of the Kostka-Foulkes polynomials $K_{\lambda,\mu}(t)$
with $\lambda \leq\nu$ due to a statistic $\mathrm{c}$ defined on the
crystals $B(\lambda )$, $\lambda\leq\nu$, and an atomic decomposition of
each $B(\lambda)$ in which (\ref{propac}) holds, we will obtain a $t$-atomic
decomposition of $B(\lambda)$ based on $\mathrm{c}$, exactly as in the proof
of Theorem~\ref{Th_tAtomic}.

Now assume that $\nu $ is a partition and $\mathfrak{g}$ is of type $C_{r}$,
with $r$ sufficiently large. In \cite{lec}, a statistic was defined on
Kashiwara-Nakashima tableaux which, conjecturally, gives a combinatorial
description of the Kostka-Foulkes polynomials.\ The definition of this
statistic is involved since it is based on the cyclage operation in
the symplectic plactic monoid, which is much more complex than in type $A$.\
It is even challenging to decide if it (or one of its versions) could satisfy
(\ref{propac}) when the atomic decomposition of $B(\lambda)$ based on the
modified crystal operators is used. In \cite{LL}, we also defined a
statistic on King tableaux of zero weight, and proved that it yields a
combinatorial description of the generalized exponents $K_{\lambda ,0}(t)$.
It is then tempting to try to connect this description with the charge
defined in \cite{lec} by using the {\em Sheats bijection}~\cite{She} between King and
Kashiwara-Nakashima tableaux, as the crystal structure is only known on the latter (an implementation of this bijection is available, see \cite{choe}). Unfortunately, it was established in \cite{GT} that this bijection does not intertwine the two statistics (a
counterexample was found in type $C_{3}$ for $\lambda =(2,1,1)$). One can also imagine
combining the Sheats bijection, the atomic decomposition of the crystal $B(\lambda)$ for $\lambda $ a partition of even size, and (\ref{propac}), in
order to define another statistic on the whole $B(\lambda)^{+}$. An
interesting question would then arise: does this new statistic give a
combinatorial description of the Lusztig $t$-analogue? If the answer is
affirmative, we would get, in type $C$: (1) a generalization of the
Lascoux-Sch\"{u}tzenberger charge; (2) an efficient algorithm, based on
crystal combinatorics, for calculating $K_{\lambda ,\mu }(t)$ starting from $%
K_{\lambda ,0}(t)$.

\subsection{Atomic decompositions for stable one-dimensional sums}\label{1ds}

One-dimensional sums are $q$-analogues of tensor product multiplicities
defined using the energy function on tensor products of {\em Kirillov-Reshetikhin (KR) 
crystals}.\ In affine type $A$, they are known to coincide with the (finite type $A$) Kostka
polynomials when the affine KR crystals considered have row or column shapes.
For the other classical affine types, one-dimensional sums corresponding to tensor
products of row and column KR crystals admit stable versions which are
special Kostka polynomials. Nevertheless, the two families do not coincide
in general. We refer the interested reader to \cite{LL2}, where we establish, in any
classical affine type, the atomic decomposition for the stable one-dimensional
sums associated to tensor products of row and column KR crystals. As
mentioned previously, we do expect that the atomic decomposition holds in
full generality, for all the Kostka-Foulkes polynomials (up to mild assumptions on the
rank of the root system considered), and not only when these polynomials coincide with
one-dimensional sums.

\section{Geometric interpretation}\label{sec:geom} 

We give an interpretation of the combinatorial atomic decomposition in terms of the {\em geometric Satake correspondence}.
For a reductive group $G$, this important theory exhibits a geometric
realization of the irreducible representation $V(\lambda)$ of highest weight
$\lambda$ of the Langlands dual group, as the {\em intersection cohomology}
$IH^{*}(\overline{{Gr}_{\lambda}})$ of the {\em Schubert variety} denoted
$\overline{{Gr}_{\lambda}}$ in the {\em affine Grassmannian} $Gr_{G}$ for $G$; there
is also a geometric basis of \emph{MV cycles} \cite{mavgld}. However, it is
hard to give concrete formulas for the MV cycles and the action. We will show
how one can understand the combinatorics of the geometric Satake correspondence via our combinatorial atomic decomposition.

The Schubert variety $\overline{Gr^\lambda}$ in $Gr_G$ has a {\em Bott-Samelson desingularization} $\widehat{\Sigma}\rightarrow \overline{Gr^\lambda}\hookrightarrow Gr_G$. Thus, we have cohomology maps
\begin{equation}\label{cohmaps}
H^*(Gr_G)\rightarrow H^*(\overline{Gr^\lambda})\hookrightarrow H^*(\widehat{\Sigma})\simeq IH^*(\overline{Gr^\lambda})\oplus \mbox{other summands}\,.
\end{equation}
The direct sum decomposition, as $H^*(Gr_G)$-modules, is given by the {\em Decomposition Theorem}, see e.g. \cite{camdtp}.

$IH^*(\overline{Gr^\lambda})$ has the {\em truncation filtration} (or {\em standard Grothendieck filtration}), which gives the Kostka-Foulkes polynomials when restricted to the weight spaces~\cite{ginpsl}. The degree $0$ piece in this filtration is the cohomology of the constant sheaf, i.e., $H^*(\overline{Gr^\lambda})$. This has the basis of classes of Schubert varieties inside $\overline{Gr^\lambda}$, which are indexed by the weights of $V(\lambda)$ considered without multiplicity, as recorded by the layer sum polynomials. In this language, the atomic decomposition decomposition in Definition~\ref{defat} is expressing the fact that there is a refinement of the truncation filtration (with the $H^*(Gr_G)$-action), whose successive quotients are isomorphic to $H^*(\overline{Gr^\mu})$ for $\mu\in P^+(\lambda)$. These quotients correspond precisely to the blocks of the partition in the combinatorial atomic decomposition, see Definition~\ref{crat}.

\end{document}